\documentclass[12pt]{amsart}
\usepackage[headings]{fullpage}
\usepackage{amssymb}
\usepackage{amsmath}
\usepackage{amsfonts}
\usepackage{amsthm}
\usepackage{multicol}
\usepackage{mathrsfs}
\usepackage{graphicx}
\usepackage[table]{xcolor}
\usepackage{xcolor}
\usepackage{enumitem}
\usepackage{url}
\usepackage[all]{xy}
\usepackage{tabularx}
\usepackage{hyperref}
\usepackage{float}
\usepackage{psfrag}
\usepackage{tikz}
\usetikzlibrary{positioning}
\setenumerate{label={$($\normalfont\arabic*)}}

\DeclareMathOperator{\St}{St}
\DeclareMathOperator{\rank}{rank}
\DeclareMathOperator{\Sup}{Supp}
\DeclareMathOperator{\Cl}{Cl}
\DeclareMathOperator{\PGL}{PGL}

\DeclareMathOperator{\Pic}{Pic}

\DeclareMathOperator{\GL}{GL}

\DeclareMathOperator{\coho}{H}
\DeclareMathOperator{\Gal}{Gal}
\DeclareMathOperator{\Res}{Res}
\DeclareMathOperator{\Jac}{Jac}

\DeclareMathOperator{\End}{End}
\DeclareMathOperator{\Aut}{Aut}

\DeclareMathOperator{\im}{Im}
\DeclareMathOperator{\Disc}{Disc}

\def\s{\mathfrak s}

\def\sm{s_{\text{max}}}

\def \QQ{\mathbb{Q}}
\def \AA{\mathbb{A}}

\def \ZZ{\mathbb{Z}}

\def \z{\mathfrak{z}}
\def \FF{\mathbb{F}}
\def \pts{\mathcal{P}}
\def \PP{\mathbb{P}}

\def\Om{\mathscr{O}}

\def \R{\mathscr{R}}

\def\C{\mathscr{C}}
\def\X{\mathscr{X}}
\def\Y{\mathscr{Y}}
\def\YY{\mathcal{Y}}
\def\D{\mathscr{D}}

\def\<#1>{{\left\langle{#1}\right\rangle}}
\def \val{v}

\def \cohoet{\text{H}^1_{\text{\'et}}}
\DeclareMathOperator{\Sp}{Sp}

\theoremstyle{plain}
\newtheorem{thm}{Theorem}[section]
\newtheorem{prop}[thm]{Proposition}

\newtheorem{lemma}[thm]{Lemma}

\theoremstyle{remark}
\newtheorem{remark}[thm]{Remark}

\newtheorem*{rem*}{Remark}
\theoremstyle{definition}
\newtheorem{obs/res}[thm]{Observation}

\newtheorem{definition}[thm]{Definition}
\newtheorem*{def*}{Definition}
\newtheorem*{ex1}{Example 1}
\newtheorem*{ex2}{Example 2}
\newtheorem*{ex3}{Example 3}

\newtheorem*{ex*}{Example}

\usepackage[OT2,OT1]{fontenc}
\usepackage{pbox}
\usepackage{tkz-graph}
\usetikzlibrary{snakes,fit,positioning,calc}

\definecolor{amethyst}{rgb}{0.6, 0.4, 0.8}
\definecolor{atomictangerine}{rgb}{1.0, 0.6, 0.4}
\definecolor{deeppeach}{rgb}{1.0, 0.8, 0.64}
\definecolor{eggshell}{rgb}{0.94, 0.92, 0.84}
\definecolor{lightapricot}{rgb}{0.99, 0.84, 0.69}
\definecolor{lemonchiffon}{rgb}{1.0, 0.98, 0.8}
\definecolor{roundabout}{rgb}{1.0, 0.91, 0.75}
\definecolor{atomictangerine}{rgb}{1.0, 0.6, 0.4}

\def\rootsep{0.03}               
\def\clustersep{0.06}            
\def\cnamescale{0.4}             
\def\cdepthscale{0.4}            
\def\cltopskip{1pt}              
\def\clbottomskip{1pt}           

\def\rootscale{0.5}   \def\rootcolor{amethyst}

\tikzset{
  root/.style = {circle,scale=\rootscale,ball color=\rootcolor},
    rc/.style 2 args = {right=#1*1.5*\clustersep of {#2.east|-first},root}, rr/.style = {right=\rootsep of {#1.east|-first},root},
  roott/.style = {circle,inner sep=-2pt,minimum size=5pt,black,font=\ttfamily\footnotesize},
    rct/.style 2 args = {right=#1*1.5*\clustersep of {#2.east|-first},roott}, rrt/.style = {right=\rootsep of {#1.east|-first},roott},
  rootA/.style = {circle,scale=\rootscaleA,ball color=\rootcolorA},
    rcA/.style 2 args = {right=#1*1.5*\clustersep of {#2.east|-first},rootA}, rrA/.style = {right=\rootsep of {#1.east|-first},rootA},
  rootB/.style = {circle,scale=\rootscaleB,ball color=\rootcolorB},
    rcB/.style 2 args = {right=#1*1.5*\clustersep of {#2.east|-first},rootB}, rrB/.style = {right=\rootsep of {#1.east|-first},rootB},
  rootC/.style = {circle,scale=\rootscaleC,ball color=\rootcolorC},
    rcC/.style 2 args = {right=#1*1.5*\clustersep of {#2.east|-first},rootC}, rrC/.style = {right=\rootsep of {#1.east|-first},rootC},
  rootD/.style = {circle,scale=\rootscaleD,ball color=\rootcolorD},
    rcD/.style 2 args = {right=#1*1.5*\clustersep of {#2.east|-first},rootD}, rrD/.style = {right=\rootsep of {#1.east|-first},rootD},
  cluster/.style = {draw=blue!70,thick,rounded corners,inner sep=22*\clustersep,outer xsep=22*\clustersep,fit=#1},
  clabel/.style  = {anchor=west,scale=\cdepthscale,black,inner sep=0,outer xsep=1,outer ysep=0},
  clabelL/.style = {above right=-\clustersep of #1t.north east,clabel},
  clabelD/.style = {below right=-\clustersep of #1t.south east,clabel},
  clouter/.style = {inner sep=0,outer sep=0,fit=#1}
}

\def\Cluster #1 = #2;{\node[cluster=#2] (#1) {};}
\def\ClusterL #1[#2] = #3;{
  \node[cluster=#3] (#1t) {}; \node[clabelL=#1] (#1l) {$#2$}; \node[clouter=(#1t)(#1l)] (#1) {};}
\def\ClusterD #1[#2] = #3;{
  \node[cluster=#3] (#1t) {}; \node[clabelD=#1] (#1d) {$#2$}; \node[clouter=(#1t)(#1d)] (#1) {};}
\def\ClusterLD #1[#2][#3] = #4;{
  \node[cluster=#4] (#1t) {}; \node[clabelL=#1] (#1l) {$#2$}; 
  \node[clabelD=#1] (#1d) {$#3$}; \node[clouter=(#1t)(#1l)(#1d)] (#1) {};}
\def\ClusterLDName #1[#2][#3][#4] = #5;{
  \node[cluster=#5] (#1t) {}; \node[clabelL=#1] (#1l) {$#2$}; 
  \node[clabelD=#1] (#1d) {$#3$}; 
  \node[scale=\cnamescale,above=\clustersep/3 of #1t,inner sep=0, outer sep=0] (#1n) {$#4$}; 
  \node[clouter=(#1l)(#1d)(#1t)] (#1) {};}

\newcommand{\Root}[4][]{
  \ifx\relax#2\relax\node[rr#1=#3] (#4) {};\else\node[rc#1={#2}{#3}] (#4) {};\fi}
\newcommand{\RootT}[5][]{
  \ifx\relax#2\relax\node[rrt#1=#3] (#4) {#5};\else\node[rct#1={#2}{#3}] (#4) {#5};\fi}

\long\def\clusterpicture#1\endclusterpicture{\pb{\vbox to \cltopskip{\vfill}\\%
  \begin{tikzpicture}\node[coordinate] (first) {};#1\end{tikzpicture}\\[-11pt]\vbox to \clbottomskip{\vfill}}}   
\long\def\clusterpictureopt#1#2\endclusterpicture{\pb{\vbox to \cltopskip{\vfill}\\%
  \begin{tikzpicture}[#1]\node[coordinate] (first) {};#2\end{tikzpicture}\\[-11pt]\vbox to \clbottomskip{\vfill}}}   

\def\pb#1{\pbox[c]{\textwidth}{\hfil #1\hfil}}

\begin{document}
\title{On Galois representations of superelliptic curves}

\author{Ariel Pacetti}
\address{FaMAF-CIEM (CONICET), Universidad Nacional de C\'ordoba,
Medina A\-llen\-de s/n, Ciudad Universitaria, 5000 C\' ordoba, Rep\'
ublica Argentina.}
\email{apacetti@famaf.unc.edu.ar}
\thanks{AP was partially supported by FonCyT BID-PICT 2018-02073}

\author{Angel Villanueva} \address{CONICET, FCEN-Universidad Nacional de
  Cuyo, Padre Jorge Contreras 1300. Parque General San Mart\'in, 5502 Mendoza, Rep\'
ublica Argentina.}
\email{angelmza1990@gmail.com}

\begin{abstract}
  A superelliptic curve over a DVR $\Om$ of residual characteristic
  $p$ is a curve given by an equation $\C:y^n=f(x)$. The purpose of
  the present article is to describe the Galois representation
  attached to such a curve under the hypothesis that $f(x)$ has all
  its roots in the fraction field of $\Om$ and that $p \nmid n$. Our
  results are inspired on the algorithm given in \cite{wewers} but our
  description is given in terms of a cluster picture as defined in
  \cite{Tim}. 
\end{abstract}				
\maketitle

\section*{Introduction}

Galois representations play a crucial role in different aspects of
modern number theory. The main source of Galois representations are
the geometric ones, namely the ones obtained by looking at the \'etale
cohomology of varieties. Among varieties, the case of curves is the
easiest one, where one can understand the local $L$-function of a
curve at a prime of good reduction by counting the number of points of
the curve's equation over different finite fields. If the curve has
bad reduction at the prime $p$, understanding the image of inertia in
the $\ell$-adic representation (for $\ell \neq p$) is much
harder. However the situation is a little better when the curve is
``semistable''. There have been very important results in this
direction during the last years (see \cite{Tim2} and
\cite{DokGalois}). Specific results for hyperelliptic curves are given
in \cite{Tim} and for the so called \emph{superelliptic curves} in
\cite{wewers}.

Let $\Om$ be a complete DVR, $\pi$ be a local uniformizer, $K$ its
field of fractions and $k$ its residue field of characteristic $p$.

\begin{def*}
  An $n$-\emph{cyclic} or \emph{superelliptic} curve is a non-singular curve obtained as
  a cyclic cover of $\PP^1$, namely it is given by an equation of the
  form
\[
\C:y^n=f(x),
  \]
  where $f(x) \in \Om[x]$, $\Disc(f(x)) \neq 0$.
\end{def*}

Computing the local $L$-factor of $\C$ for primes of good reduction
can be done quite efficiently using for example the method described
in \cite{Drew}, which gives a full description of the attached Galois
representation in this case (as inertia acts trivially). For this
reason, the main purpose of the present article is to describe the
Galois image of inertia at $p$ of the $\ell$-adic representation of a
curve $\C$ for $\ell \neq p$ (assuming $p \nmid n$). As a first step
of this goal, we restrict to the case when $f(x)$ has all its roots in
$K$ (which is not a semistable situation, but closed to it). Over an
extension $K_2$ of $K$ the stable model contains many components
$\YY_t^{(i)}$ (the notation will be explained during the article)
hence the Galois representation splits as
\begin{equation}
  \label{eq:splitting}
V_{\ell}(\Jac(\C)) \simeq \bigoplus V_{\ell}(\YY_t^{(i)}) \oplus \bigoplus_i \St(2) \otimes \chi_i,  
\end{equation}
where $\St(2)$ is the Steinberg $2$-dimensional representation, and
the $\chi_i$ are characters. To get such a representation over $K$, we
first study the stable model over $K_2$, following the results of
\cite{wewers} (where a stable model $\Y$ of $\C$ is given). Our first
contribution is to relate their description to the more combinatorial
one given in \cite{Tim} for hyperelliptic curves in terms of
\emph{clusters}.

The way to get a stable model is to start with a stable model of a
marked projective line, and compute its normalization under the
natural projection of $\C$ to $\PP^1$ (giving the stable model
$\Y$). The use of clusters provides a stable model of $\PP^1$ as
explained in \cite{Tim} and does not depend on the degree of the cover
$n$. During the exposition, we make an explicit passage from clusters
to triples as considered in \cite{wewers}. 

The new phenomena appearing when $n > 2$ is that the normalization of
projective lines might not be irreducible. This is a very interesting
phenomena that we explain in Section~\ref{section:ssmodel}. Unlike the
hyperelliptic case, components not only might have positive genus
(hence each of them contribute to the first part of
(\ref{eq:splitting})), but the different components intersect between
themselves in a very combinatorial way. Understanding the intersection
points is crucial to describe the component graph of the special fiber
of $\Y$. Furthermore, the Galois action on the intersection points
give rise to the characters appearing in the second part of
(\ref{eq:splitting}) (i.e. are the characters twisting the Steinberg
representation).

Let us explain the organization of the article. Let $\R$ denote the
set of roots of $f(x)$. In this article we assume that $\R \subset K$
and $p \nmid n$. Since we are mostly concerned with the image of
inertia (and $p \nmid n$), we also assume that $\zeta_n \in K$. Let
$K_2 = K[\sqrt[n]{\tilde{\pi}}]$, where $\tilde{\pi}$ is a local
uniformizer of $K$. In the first part of the article, we assume
furthermore that $f(x)$ is a monic polynomial, in particular
$f(x) = \prod_{r \in \R}(x-r)$.

The first section recalls the description of the stable model
$(\X,\D)$ of the marked $\PP^1$ as given in \cite{wewers}. The second
section explains its relation to the cluster picture of \cite{Tim} via
a map from $(\X,\D)$ to proper clusters. In particular, we explain how
points relate to clusters, and how the different components intersect
(Theorem~\ref{thm:geometry}).

In Section~\ref{section:ssmodel} we describe the semistable model $\Y$
over $K_2$. In particular, we give a description of the number of
components in terms of the clusters and their cardinality
(Proposition~\ref{prop:numbercomponents}). The advantage of working
over $K_2$ is that the cluster picture itself is enough to describe the
representation. The last section explains how to compute the Galois
representation over $K$. For doing that, we explain the effect of
twisting the Galois representation attached to a superelliptic
curve. To describe the twist, we explain the decomposition of
$V_\ell(\Jac(\C))$ as a module not over $\ZZ_\ell$ but over
$\ZZ_\ell[\xi_n]$. There is a contribution of the Galois
representation coming from the curves
\[
\C_{n/d}: y^{n/d} = f(x),
  \]
  for each divisor $d$ of $n$. Then our Galois representation splits
  as a sum of what might be called the $d$-new contributions. In
  Section~\ref{section:repoverK} we give a precise description of how
  to compute the twist on each $d$-new subrepresentation.

  The advantage of working with general twists is that it allows to
  consider a general polynomial $f(x)$ (not necessarily monic) over
  $K$ (with the assumption $\R \subset K$). To understand the
  restriction of the Galois representation to the inertia subgroup, it
  is very convenient to work with the notion of weighted clusters as
  recalled in the same section. The complete weighted cluster picture
  contains information on whether the characters $\chi_i$ are
  ramified or not, as explained in the last part of the article.

\subsection*{Notation} Let us give a small list of the principal symbols that will be used in the article:
\begin{itemize}
  \item By $\pi$ we denote a local uniformizer of $\Om$, $k$ its residue field and  $v(x)$ denote the $\pi$-valuation of $x$.
\item The symbol $\R$ denotes the roots of the polynomial $f(x)$.
\item $(\X,\D)$ denotes the semistable model of the marked $(\PP^1,D)$ line and $\Y$ denotes its normalization in the function field of $\C$.
\item $\overline{X}$ denotes the special fiber of $\X$ and $\overline{Y}_t^{(\ell)}$ denote the components of the special fibers of $\Y$.
\end{itemize}

\subsection*{Acknowledgments} The cluster pictures were done using the macros of \cite{Tim}. We thank Tim Dokchitser for allowing us to use this package. 

\section{The minimal stable model $(\X,\D)$}
\label{section:components}

The method to compute the stable minimal model is as follows:
consider the cover $p:\C \to \PP^1$ obtained by sending $(x,y) \to
x$. This is a cyclic Galois cover of degree $n$ ramified precisely at
the points
  \[
    D = \sum_{r \in \R}[r] + \begin{cases}
      [\infty] & \text{ if }n \nmid \deg(f(x)),\\
      0 & \text{ otherwise.}
      \end{cases}
    \]
    Consider the marked curve $(\PP^1,D)$ and compute its minimal
    semistable model $(\X,\D)$. A semistable model of $\C$ is then
    obtained as the normalization $\Y$ of $\X$ in the
    function field of $\C$, i.e. $\Y$ fits in the following diagram
    \[
\xymatrix{ 
  \Y \ar@{..>}[r]\ar@{..>}[d] & \C \ar[d]^p\\
  \X \ar[r] & \PP^1}
      \]
      The stable model $(\X,\D)$ is obtained by gluing open affine
      lines blown up at a point in the special fiber of $\PP^1$ (as
      explained in \cite[Section 3]{Tim},\cite{wewers} and also
      \cite{MR1917232}). Recall the algorithm given in \cite{wewers} to
      compute $(\X,\D)$. Let
    \begin{equation}
      \label{eq:Sdef}
      S = \Sup(D) = \R \; \cup \; \begin{cases}
        \infty & \text{ if } n \nmid \deg(f(x)),\\
        \emptyset & \text{ otherwise}.
        \end{cases}
    \end{equation}
\begin{remark}
  The assumption that $\C$ has positive genus implies in particular
  that $S$ has at least $3$ elements, as this is always the case except
  when $f(x)$ has degree $1$ or when $f(x)$ has degree $2$ and $n=2$. In both
  cases the curve $\C$ has genus $0$.
  \label{rem:3roots}
 \end{remark}
 Let $T$ denote the set of triples of distinct elements of $S$.  The
 coordinate function of $t=(a,b,c) \in T$, is defined as
\begin{equation}
  \label{eq:varphidefinition}
\varphi_t(x)=\frac{(b-c)}{(b-a)} \frac{(x-a)}{(x-c)}.
\end{equation}
It corresponds to the M\"obius transformation sending $(a,b,c)$ to
$(0,1,\infty)$. Define the following equivalence relation in $T$: two
elements $t_1, t_2 \in T$ are equivalent (which we denote
$t_1 \sim t_2$) if the map $\phi_{t_2}\circ \phi_{t_1}^{-1}$ extends
to an automorphism of $\PP^1_{\Om}$ (i.e. it corresponds to a matrix
in $\PGL_2(\Om)$).

The semistable model $\X$ consists of one component (a projective
line) for each equivalence class. Furthermore, the special fiber
$\bar{X}$ of $\X$ is a tree of projective lines where each component
contains at least $3$ points (being either elements of $S$ or
singular points were two components meet).

For each $t \in T$ the map $\varphi_t$ extends to a proper
$\Om$-morphism $\varphi_t:\X \to \PP^1_{\Om}$, whose reduction
(denoted $\overline{\varphi_t}$) is a contraction morphism with
contracts all but one component of $\bar{X}$ to a closed point (see
\cite[Proposition 4.2]{wewers}). Furthermore, if $r \in \R$ then
$\overline{\varphi_t}(r) = \overline{\varphi_t(r)}$.
  Extend the valuation $v$ on $\Om$ by setting $v(\infty) = -\infty$.
%
  
\begin{lemma}
  The equivalence relation in $T$ satisfies the following properties:
  \begin{enumerate}
  \item The permutation of a triple $(a,b,c)$ is equivalent to $(a,b,c)$.
    
  \item Any triple is equivalent to one with $v(b-c)= v(a-c) \le v(a-b)$.
  \end{enumerate}
\end{lemma}

\begin{proof}
  The first assertion follows from a straightforward matrix
  computation. For the second one, note that the map from triples of
  elements to triples of rational numbers given by
  $(a,b,c) \to (v(b-c),v(a-c),v(a-b))$ is $S_3$ invariant, hence we
  can assume $v(b-c) \le v(a-c) \le v(a-b)$. But $a-b = (a-c) + (c-b)$
  hence
  \[
    v(a-b) \ge \min\{v(a-c),v(b-c)\}=v(b-c)
  \]
  with equality if both values are different. Then the assumption
  $v(b-c) \le v(a-c) \le v(b-c)$ implies that $v(a-c) = v(b-c)$.
\end{proof}

\begin{definition}
  An \emph{ordered triple} is a triple $(a,b,c)$ with $v(a-c) = v(b-c) \le v(a-b)$.
\end{definition}

If $(a,b,c)$ is an ordered triple, define its \emph{radius} to be $\mu = v(a-b)$.

\begin{prop}
  Let $(a,b,c)$ be an ordered triple. 
  \begin{enumerate}
  \item If $\infty \in S$ then any ordered triple $(a,b,c)$ is
    equivalent to the ordered triple $(a,b,\infty)$.
    
  \item The radius $\mu$ depends only on the equivalent class of the triple.
    
  \item The ordered triple $(a,b,c)$ is equivalent to the ordered
    triple $(\alpha,\beta,\gamma)$, if and only if the following
    two properties hold:
    \begin{itemize}
    \item they have the same invariant, i.e. $\mu = v(a-b) = v(\alpha-\beta)$,
      
    \item $a \equiv b \equiv \alpha \equiv \beta \pmod{\pi^\mu}$.
    \end{itemize}
      \end{enumerate}
\label{prop:equiv}
    \end{prop}

\begin{proof}
  By the equivalence relation definition, an ordered triple
  $T_1=(a,b,c)$ (with $c \neq \infty$) is equivalent to a triple
  $T_2=(a, b,\infty)$ if and only if $\lambda_2 \circ \lambda_1^{-1}$
  extends to an automorphism of $\PP^1_{\Om}$, where $\lambda_i$ is
  the M\"obius transformation sending the triple $T_i$ to the triple
  $(0,1,\infty)$. The M\"obius matrix attached to such composition
  equals
  \[
[\lambda_2 \circ \lambda_1^{-1}] =  \frac{1}{(a-c)}\left( \begin{array}{cc}
\frac{(a-c)}{(b-c)} & 0\\ \frac{(a-b)}{(b-c)} & 1
                                           \end{array} \right).
                                       \]
       If we multiply the matrix by $(a-c)$ we get an invertible
       integral matrix, since $\frac{(a-c)}{(b-c)}$ is a unit (recall
       the definition of an ordered triple), hence both triples are
       indeed equivalent.

       To prove equivalence of ordered triples $(a,b,c)$,
       $(\alpha, \beta, \gamma)$, since we can add $\infty$ to the
       set, it is enough to restrict to the case $(a,b,\infty)$ and
       $(\alpha,\beta,\infty)$. It is easy to check that the
       transformation sending one triple to the other one is given by
       the matrix

\[
M=\left( \begin{array}{cc}
(a-b)& (\alpha-a)\\ 0  & (\alpha-\beta)
                                           \end{array} \right).
       \]
       For a multiple of $M$ to lie in $\GL_2(\Om)$, it must happen
       that $v(a-b) = v(\alpha-\beta)$, hence the two triples have the
       same \emph{radius}, as stated. At last, under such assumption,
       the two triples are equivalent if and only if
       $v(a-\alpha) \ge \mu$ (the radius of the triples). Recall that $\mu = v(a-b)$, so $a \equiv b \pmod{\pi^\mu}$ and the same holds for $\alpha$ and $\beta$, as stated.

\end{proof}

\begin{remark}
  Given a cyclic curve $\C:y^n = f(x)$, there are many transformations
  that preserve the model (for example translation). The combinatory
  behind the computation of a stable model of $\PP^1$ attached to the
  roots of $p(x)$ depends on the particular equation. However, the
  information obtained from it (number of components, discriminant,
  etc) does not.  
\end{remark}

Before stating the relation between equivalence classes of triples and clusters, let us illustrate the algorithm with an example.
\begin{ex1}
  \label{ex:1} Let $p$ be an odd prime number and $\C/\mathbb{Q}_p$ be the superelliptic curve given by the equation
\begin{small} \[ \C\!: \! y^6 \!= x(x-p^2)(x-p)(x-p-p^2)(x-2p)(x-2p-p^2)(x-1)(x-1-p)(x-1-2p) \] \end{small}
Then $\R = \lbrace 0, p^2, p, p+p^2, 2p, 2p+p^2, 1, 1+p, 1+2p \rbrace$ and $S = \R \cup \lbrace \infty \rbrace$ (since $6 \nmid \textrm{deg}(f) )$.
 By Proposition 1.4 any ordered
triple $(a,b,c)$ is equivalent to the ordered triple $(a,b,\infty)$,
and there are $36$ such triples. The radii are given in Table~\ref{table:radii}.

\begin{table}[h]
\scalebox{.7}{
\begin{tabular}{|l|c|c|c|c|c|c|c|c|c|c|c|c|c|}
\hline
  Pair & $(0,1)$ & $(0,1+p)$ & $(0,1+2p)$ & $(p^2,1)$ & $(p^2,1+p)$ & $(p^2,1+2p)$  \\
  \hline
  Radius & $0$ & $0$ & $0$ & $0$ & $0$& $0$ \\
  \hline
  Pair  & $(p,1)$ & $(p,1+p)$ & $(p,1+2p)$    &$(p+p^2,1)$ & $(p+p^2,1+p)$ &  $(p+p^2,1+2p)$ \\
  \hline
  Radius  &  $0$ & $0$ & $0$ & $0$ & $0$ & $0$\\
  \hline
  Pair & $(2p,1)$  & $(2p,1+p)$  & $(2p,1+2p)$ &  $(2p+p^2,1)$ & $(2p+p^2,1+p)$ & $(2p+p^2, 1+2p)$ \\
  \hline
  Radius  & $0$ & $0$& $0$ &  $0$ & $0$ & $0$\\
  \hline
  Pair   & $(0,p)$ & $(0,p+p^2)$ & $(0,2p)$ & $(0,2p+p^2)$ & $(p^2,p)$ & $(p^2,p+p^2)$  \\
  \hline
  Radius    & $1$ & $1$ & $1$ & $1$ & $1$& $1$ \\
  \hline
  
   Pair    & $(p^2,2p)$ & $(p^2,2p+p^2)$ & $(p,2p)$   &$(p,2p+p^2)$ & $(p+p^2,2p)$ &  $(p+p^2,2p+p^2)$ \\
  \hline
  Radius & $1$ & $1$ & $1$ & $1$ & $1$& $1$\\
  \hline
  
   Pair    & $(0,p^2)$  & $(p,p+p^2)$  & $(2p,2p+p^2)$ &  $(1,1+p)$ & $(1,1+2p)$ & $(1+p, 1+2p)$ \\
  \hline
  Radius & $2$ & $2$ & $2$ & $1$ & $1$& $1$\\
  \hline
  \end{tabular}}
\caption{}
\label{table:radii}
\end{table}

By Proposition 1.4 (3), all elements in the first three rows are equivalent, all elements in fourth and fifth rows are equivalent, the three first elements of the last row are not equivalent, and the last three elements in the last row are equivalent, hence there are six equivalent classes. The ordered
triples and the charts can be taken to be:
\begin{multicols}{2}
  \begin{itemize}
  \item $t_0=(0,1, \infty), \; \varphi_0(x)= x$
  \item $t_1=(0, p, \infty),\; \varphi_1(x)= \frac{x}{p}$
  \item $t_2=(0, p^2, \infty), \; \varphi_2(x)= \frac{x}{p^2}$
  \item $t_3=(p, p+p^2, \infty), \; \varphi_3(x)= \frac{x-p}{p^2}$ 
  \item $t_4=(2p, 2p+p^2, \infty), \;  \varphi_4(x)= \frac{x-2p}{p^2}$
    \item $t_5=(1, 1+p, \infty), \;  \varphi_4(x)= \frac{x-1}{p}$
  \end{itemize}
  \end{multicols}
%
%
Then the special fiber of $\X$ looks like Figure~\ref{fig:example1}.

\begin{figure}[h!]
\begin{tikzpicture}[xscale=1.2,yscale=1.2,
  l1/.style={shorten >=-1.3em,shorten <=-0.5em,thick}]

  \draw[l1] (0.6,0.00)--(3.60,0.00) node[scale=0.8,above left=-0.17em] {} node[scale=0.5,blue,below right=-0.5pt] {$X_m$};
\filldraw[black] (2.35,0) circle (2pt) node[scale=0.5,red,yshift=-1em] {$\infty$};
\draw[l1] (1.6,0.00)--node[right=-3pt] {} (1.6,2.0) node[xshift=.2cm, yshift=.2cm,font=\tiny, scale=0.8, blue] {$X_1$};
\draw[l1] (3.1,0.00)--node[right=-3pt] {} (3.1,2.0) node[xshift=-.2cm, yshift=.2cm,font=\tiny, scale=0.8, blue] {$X_5$};

\filldraw[black] (3.1,0.6) circle (2pt) node[scale=0.5,red,xshift=1cm] {$1+2p$};
\filldraw[black] (3.1,1.2) circle (2pt) node[scale=0.5,red,xshift=1cm] {$1+p$};
\filldraw[black] (3.1,1.8) circle (2pt) node[scale=0.5,red,xshift=1cm] {$1$};

\draw[l1] (0,0.6)--node[right=-3pt] {} (1.8,0.6) node[scale=0.8,blue,below right=-0.5pt,font=\tiny] {$X_4$};

\filldraw[black] (0.2,0.6) circle (2pt) node[scale=0.5,red,yshift=-1em] {$2p$};
\filldraw[black] (1.1,0.6) circle (2pt) node[scale=0.5,red,below] {$2p+p^2$};

\draw[l1] (0,1.2)--node[right=-3pt] {} (1.8,1.2) node[scale=0.8,blue,below right=-0.5pt,font=\tiny] {$X_3$};

\filldraw[black] (0.2,1.2) circle (2pt) node[scale=0.5,red,yshift=-1.0em] {$p$};
\filldraw[black] (1.1,1.2) circle (2pt) node[scale=0.5,red,below] {$p+p^2$};

\draw[l1] (0,1.8)--node[right=-3pt] {} (1.8,1.8) node[scale=0.8,blue,below right=-0.5pt,font=\tiny] {$X_2$};

\filldraw[black] (0.2,1.8) circle (2pt) node[scale=0.5,red,yshift=-1em] {$0$};
\filldraw[black] (1.1,1.8) circle (2pt) node[scale=0.5,red,below] {$p^2$};
\end{tikzpicture}
  \caption{Special fiber of $\X$}
  \label{fig:example1}
\end{figure}
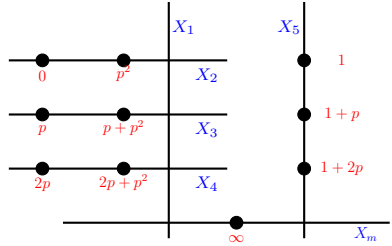
\end{ex1}
%
\section{Clusters and their relation with $\X$}
Clusters were defined in \cite{Tim} to study hyperelliptic curves. We
strongly recommend the reader to look at such article and also the
expository article \cite{2007.01749}. We follow closely the definition and
notations presented in such references. 
  \begin{definition}
    A \emph{cluster} $s$ is a non-empty subset of $\R$ of the form
    $\s = D(z,d)\cap \R$, for some disc
    $D(z,d)=\{x \in \overline{K} \; : \; v(x-z) \ge d\}$ where 
    $z \in \overline{K}$ and $d \in \QQ$. A \emph{proper} cluster
    is a cluster with more than one element.
  \end{definition}
  We will mostly be interested in proper clusters for most of the
  discussions. Let $\Cl(\R)$ denote the set of proper clusters of
  $\R$. For a cluster $\s$, let $|\s|$ denote the number of elements
  of $\R$ contained in $\s$.

  \begin{lemma}
Given $\s_1,\s_2$ clusters, then either they are disjoint or one is contained in the other. 
  \end{lemma}
  \begin{proof}
    The result follows from the fact that any point inside the disc
    defining a cluster can be taken as the ball center (see \cite[Section
    1.5]{Tim}).
  \end{proof}

  \begin{definition}
    If $\s, \s'$ are clusters with $\s' \subsetneq \s$ a maximal
    subcluster, we write $\s' < \s$ and refer to $\s'$ as a
    \emph{child} of $\s$ and $\s$ as a \emph{parent} of $\s'$. 
  \end{definition}

  To a proper cluster $\s$ we associate its \emph{diameter}
  $\mu(\s) = \min\{v(z-t) \; : \; z,t \in \s\}$ (in \cite{Tim} the
  authors use term \emph{depth} for such invariant).

  \begin{lemma}
    Let $\s$ be a proper cluster, and let $a,b \in \s$ two elements satisfying that $v(a-b) = \mu(\s)$. Then $\s = D(a,{\mu(\s)}) \cap \R$.
  \end{lemma}
  \begin{proof}
    Clearly, $v(a-b) \ge \mu(\s)$ for all $b \in \s$, hence
    $\s \subset D(a,{\mu(\s)}) \cap \R$. For the other inclusion, by
    definition $s = D(\alpha,d) \cap \R$, for some $\alpha$,
    $d$. Since $a \in \s$, $a \in D(\alpha,d)$, so we can take it as
    center, i.e. $s = D(a,d) \cap \R$. But $\mu(\s)$ is the minimal
    valuation between elements in $s$ hence $d \ge \mu(\s)$ and
    $D(a,\mu(\s)) \cap \R \subset \s$.
  \end{proof}
  \begin{definition}
  Let $\sm$ be the maximal cluster, i.e. the cluster containing
  all other clusters and all elements of $S$.  
  \end{definition}
  
 
  Let $(a,b,c)$ be an ordered triple in $T$, and let $\mu = v(a-b)$ be
  its invariant. Define a map $\Phi: T \to \Cl(\R)$ by
  \begin{equation}
    \label{eq:bijection}
    \Phi((a,b,c)) = D_{\mu}(a) \cap \R.
  \end{equation}
  \begin{thm}
    The map $\Phi$ gives a well defined map between the
    equivalence classes of triples in $T/\sim$ and the set of clusters of
    $\R$. Furthermore, the map $\Phi$ satisfies the following
    properties:
    \begin{enumerate}
      
    \item It is injective.
    \item The set $\Cl(\R) \setminus \{\sm\} \subset \im(\Phi)$.      
    \item The cluster $\sm$ lies in the image of $\Phi$ if either one of the following properties hold:
      \begin{enumerate}
      \item[i.] The element $\infty \in S$,        
      \item[ii.] there are three different elements $a,b,c \in \R$
        satisfying $$\mu(\sm) = v(a-b) = v(b-c) = v(a-c)$$ (equivalently,
        $\sm$ has more than two childs).
      \end{enumerate}
    \end{enumerate}
\label{thm:mapPhi}
  \end{thm}

  \begin{proof} For $\Phi$ to descend to a map on the quotient
    $T/\sim$, we need to prove that if $t_1=(a,b,c)$ and
    $t_2=(\alpha,\beta,\gamma)$ are equivalent ordered triples of $T$
    then $\Phi(t_1) = \Phi(t_2)$. By
    Proposition~\ref{prop:equiv}, the condition $t_1 \sim t_2$ implies
    that $\mu =v(a-b) =v(\alpha-\beta)$ and
    $a \equiv b \equiv \alpha \equiv \beta \mod \pi^\mu$. Then
    $\alpha \in D_{\mu}(a)$ hence $D_{\mu}(a) = D_{\mu}(\alpha)$ and
    $\Phi(t_1) = \Phi(t_2)$.
    \begin{enumerate}
    \item {\bf Injectivity:} let $t_1 = (a,b,c)$ and
    $t_2 = (\alpha,\beta,\gamma)$ be two ordered triples such that
    $\Phi(t_1) = \Phi(t_2)$. Note that
    \begin{equation}
      \label{eq:mufromcluster}
      v(a-b) = \mu(\Phi(t_1))=\min\{v(z-t) \; : \; z,t \in \Phi(t_1)\}.      
    \end{equation}
    Then we can recover the invariant $\mu$ of the triple $t_1$ as the
    diameter of $\Phi(t_1)$. Since $\Phi(t_1) = \Phi(t_2)$, the $\mu$
    invariant of $\Phi(t_2)$ equals that of $\Phi(t_1)$. On the other
    hand, since $\Phi(t_1) = \Phi(t_2)$,
    $\{\alpha,\beta\} \subset \Phi(t_2)$ so
    $a \equiv b \equiv \alpha \equiv \beta \pmod{\pi^\mu}$ and
    $t_1 \sim t_2$ by Proposition~\ref{prop:equiv}.
  \item Let $\s \in \Cl(\R)$ be a proper cluster which is not
    maximal. Let $a,b \in s$ be a pair such that $v(a-b) =
    \mu(s)$. Clearly $\s = D_{\mu}(a) \cap \R$. Let $c \not \in s$,
    then $\s = \Phi(a,b,c)$.

  \item As before, given $\sm$, let $a,b \in \sm$ be a pair such that
    $v(a-b) = \mu(\sm)$. 
    \begin{enumerate}
    \item [i.]  If $\infty \in S$, then $\sm = \Phi(a,b,\infty)$.
      
    \item [ii.] If there exists $a,b,c \in \sm$ with $v(a-b)=v(a-c)=v(b-c)$ then $\sm = \Phi(a,b,c)$.
    \end{enumerate}
    Finally, if $\sm$ has precisely two childs, we need to
    prove it does not lie in the image. Let $\s_1$ and $\s_2$ be the
    maximal subclusters. Any triple $(a,b,c)$ satisfies (without loss of
    generality) that two elements lie in $\s_1$ and the other in $\s_2$
    or the three of them lie in $\s_1$. In both cases it is easy to
    check that $\Phi(a,b,c) \subset \s_1$, hence $\sm$ is not in the image.
  \end{enumerate}
  \end{proof}
%
  \begin{ex2}
    \label{ex:nomaximal}
    The case when $\sm$ is not in the image of $\Phi$ corresponds to a
    model $(\X,\D)$ with precisely two lines intersecting in a single
    point. For example, if $p>3$ is a prime number, the curve with equation:
    \[
\C:y^6 = x(x-p)(x-1)(x-1+p)(x-1+2p)(x-1+3p).
    \]
    In Figure~\ref{pic:nomaximal} we present its cluster picture as well as the model $(\X,\D)$.
\begin{figure}
\centering
\begin{minipage}{.5\textwidth}
\centering
\begin{figure}[H]
\begin{tikzpicture}[xscale=1.2,yscale=1.2,
  l1/.style={shorten >=-1.3em,shorten <=-0.5em,thick}]

  \draw[l1] (0.6,0.00)--(3.60,0.00) node[scale=0.8,above left=-0.17em] {} node[scale=0.5,blue,below right=-0.5pt] {};
\filldraw[black] (2,0) circle (2pt) node[scale=0.5,red,yshift=-.5cm] {$1$};
\filldraw[black] (2.5,0) circle (2pt) node[scale=0.5,red,yshift=-0.5cm] {$1+p$};
\filldraw[black] (3,0) circle (2pt) node[scale=0.5,red,yshift=-0.5cm] {$1+2p$};
\filldraw[black] (3.5,0) circle (2pt) node[scale=0.5,red,yshift=-0.5cm] {$1+3p$};

\draw[l1] (1.6,0.00)--node[right=-3pt] {} (1.6,1.2) node[xshift=.2cm, yshift=.2cm,font=\tiny, scale=0.8, blue] {};
\filldraw[black] (1.6,0.6) circle (2pt) node[scale=0.5,red,xshift=0.5cm] {$0$};
\filldraw[black] (1.6,1.2) circle (2pt) node[scale=0.5,red,xshift=0.5cm] {$p$};

\end{tikzpicture}
\caption{Special fiber of $\X$}
\label{pic:nomaximal}
\end{figure}
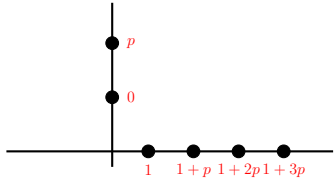
\end{minipage}%
\begin{minipage}{.5\textwidth}
\centering
\begin{figure}[H]
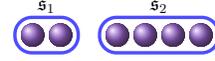

\centering
$$
\scalebox{1.6}{\clusterpicture
 \Root {1} {first} {r1};
  \Root {} {r1} {r2};
  \Root {3} {r2} {r3};
   \Root {} {r3} {r4};
    \Root {} {r4} {r5};
  \Root {} {r5} {r6}
   \ClusterLDName c1[][][\s_1] = (r1)(r2);
  \ClusterLDName c2[][][\s_2] = (r3)(r4)(r5)(r6);
\endclusterpicture}
$$
\caption{Attached clusters}
\end{figure}
\end{minipage}
\label{fig:nomaximal}
\end{figure}
\end{ex2}

Following the description in Section~\ref{section:components}, we know
that the components of $\X$ correspond to elements in $T/\sim$, hence
to get a complete description of $\X$ we need to understand how the
components intersect with each other and how the points of $S$
distribute between the components. Representing elements as clusters
gives the natural answer, namely in general two components will
intersect precisely when the cluster attached to one of them is a
child of the other. More concretely,

\begin{thm} The components attached to clusters $\s_1,\s_2$ in the
  image of $\Phi$ intersect if and only if one of the following holds:
    \begin{enumerate}
    \item     $\s_1$ is a maximal subcluster of
      $\s_2$ or vice-versa, or
      
    \item $\s_1 \cap \s_2 = \emptyset$ and both $\s_1, \s_2$ are maximal clusters.
    \end{enumerate}
      \label{thm:geometry}    
  \end{thm}

  The second case corresponds precisely to the case explained in
  Example~\ref{ex:nomaximal}. The statement is implicit in \cite{Tim}
  (see Section 5 and Theorem 1.10) as well as in \cite{wewers}, but we
  present a different proof which depends on understanding especial
  points on clusters and the coordinate functions evaluated at them.
  
  \begin{definition}
    Let $\s \in \Cl(\R)$. A \emph{point} of $\s$ is one of the following:
    \begin{itemize}
    \item a child of $\s$ (i.e. maximal subclusters of $\s$),
      
    \item a parent of $\s$ (i.e. a minimal supercluster of $\s$),
      
    \item If $\infty  \in S$, one point (denoted $\infty$) in $\sm$.
      
    \item If $\sm \not \in \im(\Phi)$ one extra point in each maximal
      cluster of $\im(\Phi)$ (corresponding to the intersection point
      of the two components, see
      Example~\ref{ex:nomaximal}).
    \end{itemize}
  \end{definition}
  Note that we did not ask the clusters to be proper while defining
  points. In particular, the elements of $S$ will be points in some
  clusters.  The component graph of the special fiber $\bar{X}$ of
  $\X$ is a stably marked tree (see section 4.2 of \cite{wewers}). In
  particular, each cluster contains at least $3$ points.
   Let $\pts$ be
  the disjoint union of points in proper clusters $\s \in \Cl(\R)$,
  i.e.
  \[
\pts = \bigsqcup_{\s \in \Cl(\R)} \{\text{points in }\s\}.
\]
In particular, the set $\R \subset \pts$.
%
%
%
%
%
    To make computations easier, we can (and will) assume that if $\s$
    is a non-maximal cluster, then the ordered triple $t=(a,b,c) \in T$
    mapping to it via $\Phi$ satisfies that its last coordinate does
    not lie in $\s$ (this is always the case up to an equivalent
    triple). Furthermore, if $\infty \in S$, we assume that
    $c = \infty$.

  \begin{lemma}
  Let $t \in T$, and $x_1,x_2 \in \R$ be roots. Let $\s = \Phi(t)$ be the
  associate cluster. Then the coordinate function $\varphi_t$ satisfies:
  \begin{enumerate}
  \item if there exists a proper subcluster $\tilde{\s} \subsetneq \s$ such that $x_1, x_2 \in \tilde{\s}$ then $\overline{\varphi_t(x_1)}=\overline{\varphi_t(x_2)}$.
  \item if $x_1, x_2 \in \s$ but they do not lie in a common maximal
    subcluster then
    $\overline{\varphi_t(x_1)}\neq \overline{\varphi_t(x_2)}$.
    
  \item if $x_1 \not \in \s$ then $\overline{\varphi_t(x_1)}=\infty$.
  \end{enumerate}
\label{lemma:rootevaluation}
\end{lemma}

\begin{proof}
  Let $t=(a,b,c)$, and consider first the case when $\Phi(t)$ is not the maximal
  cluster (hence $c \not \in \s$). By definition,
  \begin{equation}
    \label{eq:difference}
    \varphi_t(x_1)-\varphi_t(x_2) =
    \frac{(b-c)}{(b-a)}\frac{(x_1-x_2)(c-a)}{(x_1-c)(x_2-c)}.    
  \end{equation}
  \begin{enumerate}
  \item Let $x_1,x_2 \in \tilde{\s}$.  Recall that the valuation of
    the difference between one element in $\s$ and one element outside
    $\s$ is constant, then since $c \not \in \s$,
    $v(b-c) = v(c-a)=v(x_1-c)=v(x_2-c)$. The hypothesis that $x_1,x_2$
    lie in a proper subcluster implies that $v(x_1-x_2) > v(a-b) = \mu(\Phi(t))$
so the right hand side of (\ref{eq:difference}) is divisible by
    $\pi$ and $\varphi_t(x_1)\equiv \varphi_t(x_2) \pmod \pi$.
  \item If $x_1$ and $x_2$ do not lie in a proper subcluster, then
    $v(x_1-x_2) = v(a-b)$, hence the right hand side of
    (\ref{eq:difference}) is a unit, hence
    $\varphi_t(x_1) \not \equiv \varphi_t(x_2) \pmod{\pi}$.
  \item If $x_1 \not \in \s$,
    $\varphi_t(x_1)=\frac{(b-c)}{(b-a)} \frac{(x_1-a)}{(x_1-c)}$. The
    non-arquimedean triangle inequality implies that
    $v(x_1-c) = \min \{v(x_1-a), v(b-c)\}$ (recall that
    $v(a-c) = v(b-c)$). On the other hand, since $(a,b,c)$ is an
    ordered triple, $v(b-a)$ is the cluster's diameter. The hypothesis
    $x_1, c \not \in \s$ imply that $v(b-a) > \max\{v(b-c),v(x_1-a)\}$. Then
    $v(x_1-c)v(b-a) > v(x_1-a)v(b-c)$ and consequently
    $\overline{\varphi_t(x_1)} = \infty$.
  \end{enumerate}
  Assume on the contrary that $\Phi(t)$ is the maximal cluster, hence
  $v(b-c)=v(a-b)=v(a-c)$. Distinguish two cases depending on whether $x_1,x_2,c$
  belong to a common proper subcluster or not.  In the first case,
  $v(x_1-c) > v(x_1-a)$ for $i=1,2$ since $x_i$ lies in the same
  subcluster as $c$, so $\overline{\varphi_t(x_i)}=\infty$. In
  particular, if $x_1,x_2$ lie in a common proper subcluster,
  $\overline{\varphi_t(x_1)}= \overline{\varphi_t(x_2)}$.

  In the second case, if $x_1$ is not in the same proper subcluster as
  $c$, $v(x_1-c) \le v(x_1-a)$ hence
  $\overline{\varphi_t(x_1)} \neq \infty$. If $x_1,x_2$ are two roots
  not in the same proper subcluster as $c$,
  $v(b-c) =v(c-a)=v(b-a)=v(x_1-c)=v(x_2-c)$ and the same proof as
  before applies.
\end{proof}
\subsection{Functions on clusters} If $s \in \Cl(\R)$ lies in the image
of $\Phi$, define a function $\varphi_\s: \pts \to \PP^1$ extending
the coordinate function $\varphi_t$ to $\pts$ as follows.
  \begin{definition}
    Let $\s \in \Cl(\R)$ be in the image of $\Phi$, say $\s = \Phi(t)$
    and let $p \in \pts$ be a point, i.e. $p$ is a point in a cluster
    $\tilde{\s} \in \Cl(\R)$. In particular, $p$ is either a root (i.e. an element of $\R)$ or $p = \s'$ a parent/child of $\tilde{s}$. Define
  \[
    \overline{\varphi_\s}(p) =
    \begin{cases}
      \overline{\varphi_t}(\alpha) & \text{ if }p=\alpha \in \R\text{ and }\alpha \in \s,\\
      \overline{\varphi_t}(a) & \text{ if } \s'= \Phi((a,b,c)) \subset \s,\\
\infty & \text{ otherwise}.
\end{cases}
    \]
\end{definition}

\begin{remark}
  If $\infty \in S$, the point $\infty \in \sm$ evaluates to $\infty$ at all
  functions. This is clear for the function $\varphi_\s$ when $\s$ is
  not the maximal cluster, and for the maximal cluster it follows from
  the assumption $c=\infty$ of the ordered triple attached to it.
\end{remark}

\begin{lemma}
  Let $\s \in \Cl(\R)$ be an element in the image of $\Phi$.
  \begin{itemize}
  \item If $\s' \in \Cl(\R)$ is cluster not contained in  $\s$, the map
    $\varphi_\s$ takes the same value at all points of $\s'$.
    
  \item If $\s_1, \s_2$ are two different childs of $\s$, then
    $\varphi_\s$ takes different values at points of $\s_1$ and of
    $\s_2$.
  \end{itemize}

\label{lemma:contraction}
\end{lemma}

\begin{proof}
  The statements follow easily from Lemma~\ref{lemma:rootevaluation}.
\end{proof}

Suppose that $p, q \in \pts$ satisfy one of the following hypothesis:
  \begin{itemize}
  \item $\s= \s'$ and $p=q$,
    
  \item $\s$ is a child/parent of $\s'$, $p=\s'$ and $q=\s$,

  \item $\sm$ is not in the image of $\Phi$, in which case we identify
    the extra points of the two maximal proper clusters of $\im(\Phi)$
    (see Example~\ref{ex:nomaximal}).
  \end{itemize}
  Then Lemma~\ref{lemma:contraction} implies that
  $\overline{\varphi_\s}(p) = \overline{\varphi_\s}(q)$. In
  particular, all coordinate functions do not distinguish them
  (which explains why they are identified in the model $\X$). By
  definition, the function attached to a cluster $\s = \Phi(t)$ equals
  the one attached to $t$ hence they share the same properties; for
  example $\overline{\varphi_\s}$ contracts all components different
  from $t$ to points (see \cite[Proposition 4.2]{wewers}). We extend
  the function $\overline{\varphi_\s}$ to clusters defining
\[
  \overline{\varphi_\s}(\s')=
  \begin{cases}
    \infty & \text{ if } s \subset \s',\\
    \overline{\varphi_\s}(p) & \text{ otherwise }   
  \end{cases}
  \]
  \begin{prop}
Let $X_{1}, X_{2}$ be two components of $\X$. Then they do
      not intersect if and only if there exists a component $X_t$
      whose coordinate function $\overline{\varphi_t}$ collapses $X_1$
      and $X_2$ to two different points.
\label{prop:intersection}
    \end{prop}

\begin{proof}
  If $X_1$ and $X_2$ do not intersect, then since $\X$ is connected,
  there exists a component $X_t$ in the path connecting them which is
  not equal to $X_1$ nor $X_2$. Then $\overline{\varphi_t}$ collapses
  $X_1$ to one point in $\PP^1(\overline{\FF_p})$ and $X_2$ to a
  different one.
\end{proof}
We can now give a proof of how the components intersect.
\begin{proof}[Proof of Theorem \ref{thm:geometry}]
  By Proposition~\ref{prop:intersection}, the clusters $\s_i=\Phi(t_i)$
  do not intersect if and only if there exists $t_3$ such that
  $\overline{\varphi_{t_3}}$ takes different values at $\s_1$ and
  $\s_2$. Distinguish the following cases:
  \begin{itemize}
    
  \item Suppose that $\s_1 \cap \s_2 = \emptyset$ and they are maximal
    subclusters. In such case, $\sm$ is not in the image of $\Phi$. In
    particular, if $s_3$ is any other subcluster, $\s_3 \subset \s_1$ or
    $\s_3 \subset \s_2$ by Theorem~\ref{thm:mapPhi}. Then by
    Proposition~\ref{prop:intersection} $(4)$,
    $\overline{\varphi_{\s_3}}(\s_i)=\infty$ hence by
    Proposition~\ref{prop:intersection} $\s_1$ and $\s_2$ do intersect.
  \item Suppose that $\s_1 \cap \s_2 = \emptyset$ and $\s_1, \s_2$ are
    maximal subclusters of $\tilde{s}$,
    $\overline{\varphi_{\tilde{\s}}}$ takes different values at $\s_1$
    and $\s_2$, hence they do not intersect. Otherwise, let $\tilde{\s}$
    be the minimal cluster containing both of them,
    and $\tilde{\s_i}$, $i=1,2$, be a maximal subcluster of $\tilde{\s}$
    containing $\s_i$. Then the coordinate function
    $\overline{\varphi_{\tilde{\s}}}$ takes different values at $\s_1$ and $\s_2$
    by Lemma~\ref{lemma:contraction}.
    
  \item Suppose that $\s_1$ is a subcluster of $\s_2$. If it is not
    maximal, there exists $\s_3$ subcluster such that
    $\s_1 \subsetneq \s_3 \subsetneq \s_2$. Then
    $\overline{\varphi_{\s_3}}$ takes the value $\infty$ at $\s_2$ and
    sends $\s_1$ to an element in $\overline{\FF_p}$, hence they do not
    intersect.
    
  \item Suppose $\s_1$ is a maximal subcluster of $\s_2$. If they do not
    intersect, then there exists $\s_3$ cluster such that
    $\overline{\varphi_{\s_3}}$ takes different values at $\s_1$ and
    $\s_2$. Then Lemma~\ref{lemma:contraction} gives the following implications
    \begin{itemize}
    \item If $\s_1 \subset \s_2\subset \s_3$, $\overline{\varphi_{\s_3}}(\s_1) = \overline{\varphi_{\s_3}}(\s_2)$.
  \item If $\s_3 \subset \s_1 \subset \s_2$,
    $\overline{\varphi_{\s_3}}(\s_1) = \overline{\varphi_{\s_3}}(\s_2) =
    \infty$.
    
  \item If $\s_3 \cap \s_2 = \emptyset$, 
    $\overline{\varphi_{\s_3}}(\s_1) = \overline{\varphi_{\s_3}}(\s_2)$.
    \end{itemize}
    Then $\s_1$ and $\s_2$ must intersect.
  \end{itemize}
\end{proof}

\begin{ex1}[Continued] Let us go back to Example~\ref{ex:1}.
  The set of proper cluster equals:
\begin{align*}
  \s_1&=\{0, p^2, p, p+p^2,2p,2p+p^2 \},& \s_2&=\{0, p^2 \}, &\s_3&= \{p,p+p^2\},\\
  \s_4&= \{2p,2p+p^2\}, & \s_5&= \{1, 1+p, 1+2p \},& \sm&= \R.
\end{align*}
Keeping the notation of Example~\ref{ex:1}, the map $\Phi$ maps the
triples $t_i$ to $\s_i$ for $i=1,\dots, 5$ and $t_0$ to $\sm$.
The cluster picture is given in
Figure~\ref{fig:cluster1}, where the roots follow the same order as
the one given for $\R$.

\begin{figure}[H]
\centering
\begin{minipage}{.5\textwidth}
\centering
\scalebox{1.4}
  {\clusterpicture
 \Root {1} {first} {r1};
  \Root {} {r1} {r2};
  \Root {3} {r2} {r3};
  \Root {} {r3} {r4};
  \Root {3} {r4} {r5};
  \Root {} {r5} {r6};
  \Root {5} {r6} {r7};
  \Root {} {r7} {r8};
  \Root {} {r8} {r9}
   \ClusterLDName c1[][][\s_2] = (r1)(r2);
  \ClusterLDName c2[][][\s_3] = (r3)(r4);
   \ClusterLDName c3[][][\s_4] = (r5)(r6);
   \ClusterLDName c4[][][\s_1] = (c1)(c2)(c3)(c1n)(c2n)(c3n);
  \ClusterLDName c5[][][\s_5] = (r7)(r8)(r9);
  \ClusterLDName c6[][][\sm] = (c4)(c5)(c4n)(c5n);
 \endclusterpicture}
 \caption{Cluster of $\C$.}
\label{fig:cluster1}
\end{minipage}%
\begin{minipage}{.5\textwidth}
\centering
\scalebox{1}
{\begin{tikzpicture}[xscale=1.2,yscale=1.2,
  l1/.style={shorten >=-1.3em,shorten <=-0.5em,thick}]
  \draw[l1] (0.6,0.00)--(3.60,0.00) node[scale=0.8,above left=-0.17em] {} node[scale=0.5,blue,below right=-0.5pt] {$X_m$};
\filldraw[black] (2.35,0) circle (2pt) node[scale=0.5,red,yshift=-1em] {$\infty$};
\draw[l1] (1.6,0.00)--node[right=-3pt] {} (1.6,2.0) node[xshift=.2cm, yshift=.2cm,font=\tiny, scale=0.8, blue] {$X_1$};
\draw[l1] (3.1,0.00)--node[right=-3pt] {} (3.1,2.0) node[xshift=-.2cm, yshift=.2cm,font=\tiny, scale=0.8, blue] {$X_5$};

\filldraw[black] (3.1,0.6) circle (2pt) node[scale=0.5,red,xshift=1cm] {$1+2p$};
\filldraw[black] (3.1,1.2) circle (2pt) node[scale=0.5,red,xshift=1cm] {$1+p$};
\filldraw[black] (3.1,1.8) circle (2pt) node[scale=0.5,red,xshift=1cm] {$1$};

\draw[l1] (0,0.6)--node[right=-3pt] {} (1.8,0.6) node[scale=0.8,blue,below right=-0.5pt,font=\tiny] {$X_4$};

\filldraw[black] (0.2,0.6) circle (2pt) node[scale=0.5,red,yshift=-1em] {$2p$};
\filldraw[black] (1.1,0.6) circle (2pt) node[scale=0.5,red,below] {$2p+p^2$};

\draw[l1] (0,1.2)--node[right=-3pt] {} (1.8,1.2) node[scale=0.8,blue,below right=-0.5pt,font=\tiny] {$X_3$};

\filldraw[black] (0.2,1.2) circle (2pt) node[scale=0.5,red,yshift=-1.0em] {$p$};
\filldraw[black] (1.1,1.2) circle (2pt) node[scale=0.5,red,below] {$p+p^2$};

\draw[l1] (0,1.8)--node[right=-3pt] {} (1.8,1.8) node[scale=0.8,blue,below right=-0.5pt,font=\tiny] {$X_2$};

\filldraw[black] (0.2,1.8) circle (2pt) node[scale=0.5,red,yshift=-1em] {$0$};
\filldraw[black] (1.1,1.8) circle (2pt) node[scale=0.5,red,below] {$p^2$};
\end{tikzpicture}}
\end{minipage}
\end{figure}


\end{ex1}

\section{The semistable model}
\label{section:ssmodel}
Recall that $\C$ is a cyclic cover of $\PP^1$, coming from the natural
map $\phi:\C \to \AA^1$ sending $(x,y)$ to $x$. Let $K_2=K(\pi^{1/d})$
and $\Y$ be the normalization of $\X$ in the function field of
$Y_{K_2}$. The hypothesis $\R \subset K$ and $\pi^{1/d} \in K_2$ implies
that $\Y$ is a semistable model of $Y$ (by \cite[Corollary
3.6]{wewers}).  The special fiber $\bar{Y}$ of $\Y$ is obtained as
follows: let $t \in T$ with $\Phi(t) = \s =D(r,d)$ (we can assume
$r \in \R$). The cluster $\s$ corresponds to a component of the
special fiber of $\bar{X}$. Let $x_t :=\phi_t^*(x)$ be the pullback of
the standard coordinate $x$ of $X$. The two variables are related via
$x = \pi^dx_t+r$. Consider the polynomial $f(x_t)$ and let $e_t$ be
its content valuation (in Section \ref{section:weightedcluster} we
will explain how to compute such value from a \emph{weighted
  cluster}). Let $f_t(x_t) = f(x_t)\pi^{-e_t}$ and define the curve:
\begin{equation}
  \label{eq:fibers}
\overline{\YY}_t: y_t^n = \overline{f_t}(x_t).
\end{equation}
The curve $\overline{Y_t}$ is defined as the normalization of
$\overline{\YY}_t$. Note that the curve $\overline{\YY_t}$ might be
reducible, and even its components might not be defined over $K$
(but they are over an unramified extension of degree at most $n$).
Explicitly, let
$\tilde{\s}_1, \ldots, \tilde{\s}_N$ be the children of $\s$. Let
$\alpha_i \in \tilde{\s_i}$ be any root and let $a_i = |\tilde{\s}_i|$. Each
cluster $\tilde{\s}_i$ correspond to a factor of $\overline{f_t}$ and
the number $a_i$ to the multiplicity of the root $\alpha_i$. Let also
$c_t = \prod_{\beta \in\R \setminus \{\s\}}\frac{(r-\beta)}{\|(r-\beta)\|_p}$.

\begin{prop}
\label{prop:numbercomponents}
Let $d:=gcd(n,a_1, \ldots, a_N)$ and keep the previous notation.  Then
the curve $\overline{\YY_t}$ has $d$ irreducible components defined
over the extension $K(c_t^{1/d})$. In particular the same holds for
$\overline{Y_t}$.
\end{prop}
\begin{proof} If $\beta \in \R$ is a root not contained in $\s$ then
  the term $x-\beta = \pi^dx_t+r - \beta$ reduces to
  $\frac{(r-\beta)}{\|r-\beta\|_p}$ up to a power of $\pi$ (which can
  be removed from the equation by the assumption $\pi^{1/d} \in
  K_2$). Then the reduction of the polynomial $f_t(x_t)$ equals
  $\overline{f_t}(x_t)=c_t \prod_{i=1}^N(x_t-\alpha_i)^{a_i}$. For $\ell = 0,\ldots,d-1$ let
  \begin{equation}
    \label{eq:components}
\overline{\YY_t^{(\ell)}}:y_t^{n/d}=\zeta_d^\ell c_t^{1/d}\prod_{i=1}^N (x_t-\alpha_i)^{a_i/d},    
  \end{equation}
  where $\zeta_d$ denotes a $d$-th root of unity. Then clearly
  $\overline{\YY}_t = \prod_{\ell=0}^{d-1} \overline{\YY_t^{(\ell)}}$. Each
  curve $\overline{\YY_t^{(\ell)}}$ is irreducible, because the cover
  $K[x_t,y_t]/(y_t^n-f_t(x_t))$ of $K[x]$ is Galois, hence the
  ramification degree is the same on all its components. In
  particular, the number of components divide $a_i$ for all $i$.
 \end{proof}
 To understand the semistable model $\Y$ we are led to describe how
 different components intersect.  If $P$ is a point in
 $\overline{X}_t$, then number of points of $\varphi_t^{-1}(P)$ in
 $\overline{Y_t}$ equals $r={\gcd(n,v(f_t))}$ each of them with
 ramification degree $\frac{n}{\gcd(n,v(f_t))}$. In particular, each
 component gets $\frac{\gcd(n,v(f_t))}{d}$ different points.

 If $P \in \R$ (to easy notation suppose that $P=\alpha_1=0$, which
 can always be done after a translation) then the normalization
 of (\ref{eq:fibers}) in an open set around $0$ is given by the equations
 \[
   \left\{
\begin{aligned}   
    z_t^r &= c_t\prod_{i=2}^N(x_t-\alpha_i)^{a_i}\\
    z_tx_t^{a_1/r}&=y_t^{n/r}.
  \end{aligned}
  \right.
\]
In particular, the set of $r$ points in $\varphi_t^{-1}(0)$, with
coordinates $(x_t,y_t,z_t)$ is given by
\[
  \left\{Q_i=\left(0,0,\zeta_r^i\left(c_t\prod_{i=2}^N(-\alpha_i)^{a_i}\right)^{1/r}\right)\; : \; 0 \le
  i \le r-1\right\}.
  \]

Recall that $d = \gcd(n,a_1,\ldots,a_N)$, in particular
$d \mid r = \gcd(n,a_1)$. From the decomposition
(\ref{eq:components}), and choosing the roots of unity in a consistent
way (i.e. such that they satisfy $\zeta_{nm}^n = \zeta_m$) it follows that
$Q_i \in \overline{Y_t^{(\ell)}}$ precisely when
$i \equiv \ell \pmod d$. In particular, $Q_0$ belongs to the zeroth
curve, $Q_1$ to the first one, and so on.

\medskip

Let $\tilde{\s}$ be a child of $\s$ (corresponding to
$\tilde{t} \in T$); say $\tilde{\s}=\tilde{\s}_1$ in the above
notation and the center is again $0$. Then the curve
$\overline{\YY_{\tilde{t}}}$ has an equation as in (\ref{eq:fibers}),
more concretely, let $\tilde{c} = c_t\prod_{i=2}^N(-\alpha_i)^{a_i}$, then
\begin{equation}
  \label{eq:yychild}
  \overline{\YY_{\tilde{t}}}: y_{\tilde{t}}^n = \tilde{c}
  \prod_{i=1}^{\tilde{N}}(x_{\tilde{t}}-\beta_i)^{b_i},
\end{equation}
where the product is over childs of $\tilde{\s}$, the numbers $b_i$
equal the number of roots in each child, and $\beta_i$ is a root in
each of them. Note that $\sum_{i=1}^{\tilde{N}}b_i = a_1$. The gluing
(as described in \cite{Tim} before Remark 3.9) corresponds in our
coordinates to identify the infinity point in the chart $\tilde{t}$
with the zero point in the chart $t$. Then equation~(\ref{eq:yychild})
can be written as
\begin{equation}
  \label{eq:curve2}
  \left(\frac{y_{\tilde{t}}^{n/r}}{x_{\tilde{t}}^{a_1/r}}\right)^r = \tilde{c} \prod_{i=1}^{\tilde{N}}\left(1-\frac{\beta_i}{x_{\tilde{t}}}\right)^{b_i}.
\end{equation}
  This equation is the key to identify the points $Q_i$ in $\s$ with
  $r$-points in (\ref{eq:curve2}) (or its components if it happens to be
  reducible) and gives the intersection points of $\s$ and
  $\tilde{\s}$ (see the formulas in \cite[Proposition 5.5]{Tim}).

  Let $\tilde{d} = \gcd(n,b_1,\ldots,b_{\tilde{N}})$, then the curve
  $\overline{\YY_{\tilde{t}}}$ consists on $\tilde{d}$ components
  (ordered according to powers of $\tilde{d}$-th roots of unity) and
  with the compatible choice of roots of unity, the point $Q_0$ lies
  at the infinity part of the zeroth component, and so on. Let us
  illustrate the situation with some examples.
  
  \begin{ex1}[Continued II] Recall that there are six components (see
    Figures~\ref{fig:example1} and \ref{fig:cluster1}):
\begin{itemize}
\item $\overline{\mathcal{Y}}_m: y_m^6 = x_m^6 (x_m - 1)^3$ \qquad (with relations $x_m = x$, ),
\item $ \overline{\mathcal{Y}}_1: y_1^6 = (-1)x_1^2 (x_1 - 1)^2(x_1 - 2)^2$ \qquad (with relations $x_1 = x/p$, $y_1 = y/p$),
\item $\overline{\mathcal{Y}}_2: y_2^6 = (-4)x_2 (x_2 - 1)$ \qquad (with relations $x_2 = x/p^2$, $y_2 = y/p^{4/3}$),
\item $\overline{\mathcal{Y}}_3: y_3^6 = (-4)x_3 (x_3 - 1)$ \qquad (with relations $x_3 = (x-p)/p^2$, $y_3 = y/p^{4/3}$),
\item $\overline{\mathcal{Y}}_4: y_4^6 = (-4)x_4 (x_4 - 1)$ \qquad (with relations $x_4 = (x-2p)/p^2$, $y_4 = y/p^{4/3}$),
\item $\overline{\mathcal{Y}}_5: y_5^6 = x_5 (x_5 - 1)(x_5 -2)$ \qquad (with relations $x_5 = (x-1)/p$, $y_5 = y/p^{1/2}$).
\end{itemize}
The curves
$\overline{\mathcal{Y}}_2, \overline{\mathcal{Y}}_3,
\overline{\mathcal{Y}}_4, $ are non-singular irreducible curves of
genus $2$, while $\overline{\mathcal{Y}}_5$ is a non-singular curve of
genus $4$. On the other hand, the curves $\overline{\mathcal{Y}}_m$
and $\overline{\mathcal{Y}}_1$ are reducible. The curve
$\overline{\mathcal{Y}}_m$ consists of the union of three (genus $0$)
curves $\overline{\mathcal{Y}_m^{(\ell)}}$, $\ell=0,1,2$ with
equations
\[
  \overline{\mathcal{Y}_m^{(\ell)}}: y_m^2 = \zeta_3^\ell x_m^2(x_m -1) \; (y_m = y),
  \]
  where $\zeta_3$ is a third root of unity in $\FF_p$. The curve
  $\overline{\mathcal{Y}}_1$ consists of the union of two genus $1$
  curves $\overline{\mathcal{Y}_1^{(\ell)}}$, $\ell=0,1$ with equation
\[
  \overline{\mathcal{Y}_1^{(\ell)}}: y_1^3 =(-1)^\ell \sqrt{-1} x_1(x_1 -1)(x_1 - 2) \; (y_1 = y/p).
\]
Note that the components of $\overline{\mathcal{Y}_1}$ need
not be defined over $K_2$, but at most over an unramified extension
(since $p\nmid 6$, $K_2(\sqrt{-1})/K_2$ is unramified).  The
normalization explained in the previous section, in an open
neighborhood of $0$ (but not of $1$) of the curve
$\overline{\mathcal{Y}_m^{(\ell)}}$ has equation
\[
\begin{cases} z_m^2 = \zeta_3^\ell(x_m -1) \\
  z_m x_m = y_m. \end{cases}
\]
  The preimage of $0$ in the $\ell$-th component corresponds to the
  points $P_\ell^{\pm}=\left(0,0, \pm \sqrt{-1}\zeta_3^{2\ell}\right)$. In
  particular, it intersects $\overline{\mathcal{Y}}_1$ in $2$ points.
  The component graph of the special fiber of $\Y$ is given in
  Figure~\ref{fig:componentgraph}.

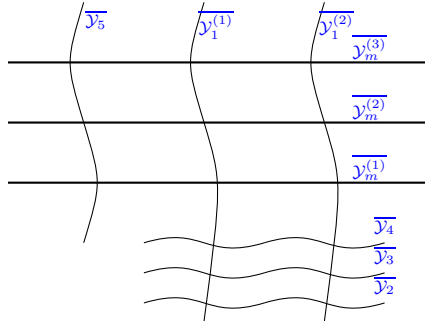
\begin{figure}[h]
\begin{tikzpicture}[xscale=.8,yscale=.8,
  l1/.style={shorten >=-1.3em,shorten <=-0.5em,thick}]
  \draw[l1] (0,1)--node[right=-3pt] {} (6,1) node[xshift=-.2cm, yshift=.2cm,font=\tiny, scale=0.8, blue] {$\overline{\YY_m^{(1)}}$};
  \draw[l1] (0,2)--node[right=-3pt] {} (6,2) node[xshift=-.2cm, yshift=.2cm,font=\tiny, scale=0.8, blue] {$\overline{\YY_m^{(2)}}$};
  \draw[l1] (0,3)--node[right=-3pt] {} (6,3) node[xshift=-.2cm, yshift=.2cm,font=\tiny, scale=0.8, blue] {$\overline{\YY_m^{(3)}}$};
  \draw (1,0) .. controls (1.3,1) and (1.3,1).. (1,2) ;
  \draw (1,2) .. controls (0.7,3) and (0.7,3) .. (1,4);

  \node at (1.2,4) [font=\tiny, scale=0.8, blue,below] {$\overline{\mathcal{Y}_5}$};

  \draw (3,-1.3) .. controls (3.3,1) and (3.3,1).. (3,2) ;
  \draw (3,2) .. controls (2.7,3) and (2.7,3) .. (3,4);

  \node at (3.2,4) [font=\tiny, scale=0.8, blue,below] {$\overline{\mathcal{Y}_1^{(1)}}$};

  \draw (5,-1.3) .. controls (5.3,1) and (5.3,1).. (5,2) ;
  \draw (5,2) .. controls (4.7,3) and (4.7,3) .. (5,4) ;

  \node at (5.2,4) [font=\tiny, scale=0.8, blue,below] {$\overline{\mathcal{Y}_1^{(2)}}$};

  \draw (2,0) .. controls (3,0.3) and (3,-0.3).. (4,0) ;
  \draw (4,0) .. controls (5,0.3) and (5,-0.3) .. (6,0);

  \node at (6,0) [font=\tiny, scale=0.8, blue,above] {$\overline{\YY_4}$};
  
  \draw (2,-.5) .. controls (3,-0.2) and (3,-0.8).. (4,-0.5) ;
  \draw (4,-0.5) .. controls (5,-0.2) and (5,-0.8) .. (6,-0.5);

  \node at (6,-0.5) [font=\tiny, scale=0.8, blue,above] {$\overline{\YY_3}$};
  
  \draw (2,-1) .. controls (3,-0.7) and (3,-1.3).. (4,-1) ;
  \draw (4,-1) .. controls (5,-0.7) and (5,-1.3) .. (6,-1);

\node at (6,-1) [font=\tiny, scale=0.8, blue,above] {$\overline{\YY_2}$};
\end{tikzpicture}
\caption{Special Fiber of $\Y$}
\label{fig:componentgraph}
\end{figure}
\end{ex1}

\begin{ex3}
  \label{example:2}
  Let $p$ be an odd prime greater than $3$ and $\C/\ZZ_p$ be the curve with equation
\begin{multline*}
y^6=x(x-p^2)(x-p)(x-p-p^2)(x-2p)(x-2p-p^2)(x-1)(x-1-p^2)(x-1-2p^2)\\(x-1-p)(x-1-p-p^2)(x-1-p-2p^2)(x-2)(x-2-p)(x-2-2p).
\end{multline*}
It is a curve of genus $34$.  The set of roots of $f(x)$ equals
$\R=\{0,p,p^2,p+p^2,2p,2p+p^2,1,1+p,1+p^2,1+2p^2,1+p+p^2,1+p+2p^2,2,2+p,2+2p\}$. There
are nine clusters as shown in Figure~\ref{fig:cluster5}. They give the components:
\begin{figure}[h]
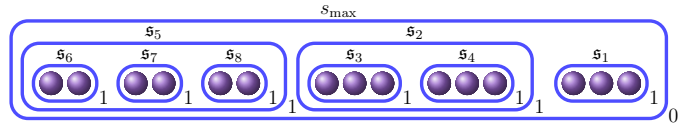

$$
\scalebox{1.6}{\clusterpicture
 \Root {1} {first} {r1};
  \Root {} {r1} {r2};
  \Root {3} {r2} {r3};
  \Root {} {r3} {r4};
  \Root {3} {r4} {r5};
  \Root {} {r5} {r6};
  \Root {5} {r6} {r7};
  \Root {} {r7} {r8};
  \Root {} {r8} {r9};
  \Root {3} {r9} {r10};
  \Root {} {r10} {r11};
  \Root {} {r11} {r12};
  \Root {5} {r12} {r13};
  \Root {} {r13} {r14};  
  \Root {} {r14} {r15}
   \ClusterLDName c1[][1][\s_6] = (r1)(r2);
  \ClusterLDName c2[][1][\s_7] = (r3)(r4);
   \ClusterLDName c3[][1][\s_8] = (r5)(r6);
   \ClusterLDName c4[][1][\s_5] = (c1)(c2)(c3)(c1n)(c2n)(c3n);
  \ClusterLDName c5[][1][\s_3] = (r7)(r8)(r9);
   \ClusterLDName c6[][1][\s_4] = (r10)(r11)(r12);
    \ClusterLDName c7[][1][\s_2] = (c5)(c6)(c5n)(c6n);
     \ClusterLDName c8[][1][\s_1] = (r13)(r14)(r15);
  \ClusterLDName c9[][0][\sm] = (c4)(c7)(c8)(c4n)(c7n)(c8n);
\endclusterpicture}
$$
\caption{Cluster picture}
\label{fig:cluster5}
\end{figure}

\begin{itemize}
\item $\sm$ is the disc with center $r_m=0$ and diameter $\mu=0$. It
  corresponds to a component
  $\overline{\YY_m}:y_m^6 =x_m^6(x_m-1)^6(x_m-2)^3$ consisting of $3$
  irreducible components
  $\overline{\YY_m^{(i)}}:y_m^2=\zeta_3^ix_m^2(x_m-1)^2(x_m-2)$,
  $0 \le i \le 2$ of genus $0$ (see Proposition~\ref{prop:genus}).
\item $\s_1=\{2,2+p,2+2p\}=D(2,1)$, with variable $x=px_1+2$,
  $y=p^{1/2}y_1$ and equation
  $\overline{\YY_1}:y_1^6=2^6x_1(x_1-1)(x_1-2)$. It is an irreducible
  curve of genus $4$.
\item $\s_2=\{1,1+p,1+p^2,1+2p^2,1+p+p^2,1+p+2p^2\} = D(1,1)$, with
  variable $x=px_2+1$, $y=py_2$ and equation
  $\overline{\YY_2}:y_2^6=-x_2^3(x_2-1)^3$. It consists of three
  irreducible components
  $\overline{\YY_2^{(i)}}:y_1^2=-\zeta_3^ix_1(x_1-1)$, $0 \le i \le 2$
  of genus $0$.
\item $\s_3=\{1,1+p^2,1+2p^2\} = D(1,2)$, with variable $x=p^2x_3+1$,
  $y = p^{3/2}y_3$ and equation
  $\overline{\YY_3}:y_3^6=x_3(x_3-1)(x_3-2)$. It is an irreducible
  curve of genus $4$.
\item $\s_4=\{1+p,1+p+p^2,1+p+2p^2\} = D(1+p,2)$, with variable
  $x=p^2x_4+1+p$, $y = p^{3/2}y_4$ and equation
  $\overline{\YY_4}:y_4^6=-x_4(x_4-1)(x_4-2)$. It is an irreducible
  curve of genus $4$.
\item $\s_5=\{0,p,p^2,p+p^2,2p,2p+2p^2\} = D(0,1)$, with variable
  $x=px_5$, $y = py_5$ and equation
  $\overline{\YY_5}:y_5^6=-8x_5^2(x_5-1)^2(x_5-2)^2$. It consists of
  two irreducible components
  $\overline{\YY_5^{(i)}}:y_5^3=(-1)^i2\sqrt{-2}x_5(x_5-1)(x_5-2)$,
  $i=0, 1$ of genus $1$.
\item $\s_6=\{0,p^2\} = D(0,2)$, with variable $x=p^2x_6$,
  $y = p^{4/3}y_6$ and equation
  $\overline{\YY_6}:y_6^6=-32x_6(x_6-1)$. It is an irreducible curve of genus $2$.
\item $\s_7=\{p,p+p^2\} = D(p,2)$, with variable $x=p^2x_7+p$,
  $y = p^{4/3}y_7$ and equation
  $\overline{\YY_7}:y_7^6=-8x_7(x_7-1)$. It is an irreducible curve of genus
  $2$.
\item $\s_8=\{2p,2p+p^2\} = D(2p,2)$, with variable $x=p^2x_8+2p$,
  $y = p^{4/3}y_8$ and equation
  $\overline{\YY_8}:y_8^6=-32x_8(x_8-1)$. It is an irreducible curve of genus
  $2$.
\end{itemize}
The special fiber of $\X$ and $\Y$ are given in Figure~\ref{pic:Yt} and Figure~\ref{pic:Yspecialfiber} respectively.
\begin{figure}[h]
\centering
\begin{minipage}{.5\textwidth}
  \centering
\begin{tikzpicture}[xscale=1.2,yscale=1.2,
  l1/.style={shorten >=-1.3em,shorten <=-0.5em,thick}]
  \draw[l1] (0.6,0.00)--(3.60,0.00) node[scale=0.8,above left=-0.17em] {} node[scale=0.5,blue,below right=-0.5pt] {$X_m$};
\filldraw[black] (1,0) circle (2pt) node[scale=0.5,red,yshift=-1em] {$\infty$};
\draw[l1] (1.6,0.00)--node[right=-3pt] {} (1.6,2.0) node[xshift=.2cm, yshift=.2cm,font=\tiny, scale=0.8, blue] {$X_5$};
\draw[l1] (3.1,0.00)--node[right=-3pt] {} (3.1,2.0) node[xshift=-.2cm, yshift=.2cm,font=\tiny, scale=0.8, blue] {$X_2$};

\draw[l1] (2.35,0.00)--node[right=-3pt] {} (2.35,2.0) node[xshift=-.2cm, yshift=.2cm,font=\tiny, scale=0.8, blue] {$X_1$};

\filldraw[black] (2.35,0.4) circle (2pt) node[scale=0.5,red,xshift=1cm] {$1+2p$};
\filldraw[black] (2.35,2) circle (2pt) node[scale=0.5,red,xshift=1cm] {$1+p$};
\filldraw[black] (2.35,1.4) circle (2pt) node[scale=0.5,red,xshift=.5cm] {$1$};

\draw[l1] (0,0.6)--node[right=-3pt] {} (1.8,0.6) node[scale=0.8,blue,below right=-0.5pt,font=\tiny] {$X_8$};

\filldraw[black] (0.2,0.6) circle (2pt) node[scale=0.5,red,yshift=-1em] {$2p$};
\filldraw[black] (1.1,0.6) circle (2pt) node[scale=0.5,red,below] {$2p+p^2$};

\draw[l1] (0,1.2)--node[right=-3pt] {} (1.8,1.2) node[scale=0.8,blue,below right=-0.5pt,font=\tiny] {$X_7$};

\filldraw[black] (0.2,1.2) circle (2pt) node[scale=0.5,red,yshift=-1.0em] {$p$};
\filldraw[black] (1.1,1.2) circle (2pt) node[scale=0.5,red,below] {$p+p^2$};

\draw[l1] (0,1.8)--node[right=-3pt] {} (1.8,1.8) node[scale=0.8,blue,below right=-0.5pt,font=\tiny] {$X_6$};

\filldraw[black] (0.2,1.8) circle (2pt) node[scale=0.5,red,yshift=-1em] {$0$};
\filldraw[black] (1.1,1.8) circle (2pt) node[scale=0.5,red,below] {$p^2$};

  \draw[l1] (2.8,1.5)--(4.60,1.5) node[scale=0.8,above left=-0.17em] {} node[scale=0.5,blue,xshift=-4cm, yshift=-1em] {$X_3$};
\filldraw[black] (3.4,1.5) circle (2pt) node[scale=0.5,red,yshift=-1em] {$1$};
\filldraw[black] (4.1,1.5) circle (2pt) node[scale=0.5,red,yshift=-1em] {$1+p^2$};
\filldraw[black] (4.9,1.5) circle (2pt) node[scale=0.5,red,yshift=-1em] {$1+2p^2$};

  \draw[l1] (2.8,0.9)--(4.60,0.9) node[scale=0.8,above left=-0.17em] {} node[scale=0.5,blue,xshift=-4cm, yshift=-1em] {$X_4$};
\filldraw[black] (3.4,0.9) circle (2pt) node[scale=0.5,red,yshift=-1em] {$1+p$};
\filldraw[black] (4.1,0.9) circle (2pt) node[scale=0.5,red,yshift=-1em] {$1+p+p^2$};
\filldraw[black] (4.9,0.9) circle (2pt) node[scale=0.5,red,yshift=-1em] {$1+p+2p^2$};
\end{tikzpicture}
\caption{Special Fiber of $\X$}
\label{pic:Yt}
\end{minipage}%
\begin{minipage}{.5\textwidth}
\centering
\begin{tikzpicture}[xscale=.8,yscale=.8,
  l1/.style={shorten >=-1.3em,shorten <=-0.5em,thick}]
  \draw[l1] (0,1)--node[right=-3pt] {} (5,1) node[xshift=-4.5cm, yshift=.2cm,font=\tiny, scale=0.8, blue] {$\overline{\YY_m^{(2)}}$};
  \draw[l1] (0,2)--node[right=-3pt] {} (5,2) node[xshift=-4.5cm, yshift=.2cm,font=\tiny, scale=0.8, blue] {$\overline{\YY_m^{(1)}}$};
  \draw[l1] (0,3)--node[right=-3pt] {} (5,3) node[xshift=-4.5cm, yshift=.2cm,font=\tiny, scale=0.8, blue] {$\overline{\YY_m^{(0)}}$};
  \draw (1,0) .. controls (1.3,1) and (1.3,1).. (1,2.6) ;
  \draw (1,2.6) .. controls (0.7,3.6) and (0.7,3.6) .. (1,5.3);

  \node at (0.7,.5) [font=\tiny, scale=0.8, blue,below] {$\overline{\mathcal{Y}_5^{(1)}}$};
  \draw (2,0) .. controls (2.3,1) and (2.3,1).. (2,2.6) ;
  \draw (2,2.6) .. controls (1.7,3.6) and (1.7,3.6) .. (2,5.3);

  \node at (1.7,.5) [font=\tiny, scale=0.8, blue,below] {$\overline{\mathcal{Y}_5^{(2)}}$};

  \draw (3.5,0) .. controls (3.8,1) and (3.8,1).. (3.5,2.6) ;
  \draw (3.5,2.6) .. controls (3.2,3.6) and (3.2,3.6) .. (3.5,4) ;

  \node at (3.2,.5) [font=\tiny, scale=0.8, blue,below] {$\overline{\mathcal{Y}_1}$};

  \draw (0,4) .. controls (0.8,3.7) and (0.8,3.7).. (1.5,4) ;
  \draw (1.5,4) .. controls (2.2,4.3) and (2.2,3.7) .. (3,4);

  \node at (0.2,3.8) [font=\tiny, scale=0.8, blue,above] {$\overline{\YY_6}$};

    \draw (0,4.5) .. controls (0.8,4.2) and (0.8,4.2).. (1.5,4.5) ;
  \draw (1.5,4.5) .. controls (2.2,4.8) and (2.2,4.2) .. (3,4.5);

  \node at (0.2,4.3) [font=\tiny, scale=0.8, blue,above] {$\overline{\YY_7}$};

  \draw (0,5) .. controls (0.8,4.7) and (0.8,4.7).. (1.5,5) ;
  \draw (1.5,5) .. controls (2.2,5.3) and (2.2,4.7) .. (3,5);

  \node at (0.2,4.8) [font=\tiny, scale=0.8, blue,above] {$\overline{\YY_8}$};

  \draw (4.3,1.2) .. controls (4.8,0.7) and (4.8,0.7).. (5.3,1.4) ;
  \draw (5.3,1.4) .. controls (5.7,1.82) and (5.7,1.7).. (6.1,1.4) ;
  \draw (6.1,1.4) .. controls (6.5,1.1) and (6.5,1.1) .. (7,1.4);

  \node at (4.2,2.2) [font=\tiny, scale=0.8, blue,above] {$\overline{\YY_2^{(1)}}$};
  
  \draw (4.3,2.2) .. controls (4.8,1.7) and (4.8,1.7).. (5.3,2.4) ;
  \draw (5.3,2.4) .. controls (5.7,2.82) and (5.7,2.7).. (6.1,2.4) ;
  \draw (6.1,2.4) .. controls (6.5,2.1) and (6.5,2.1) .. (7,2.4);

  \node at (4.2,1.2) [font=\tiny, scale=0.8, blue,above] {$\overline{\YY_2^{(2)}}$};
  \draw (4.3,3.2) .. controls (4.8,2.7) and (4.8,2.7).. (5.3,3.4) ;
  \draw (5.3,3.4) .. controls (5.7,3.82) and (5.7,3.7).. (6.1,3.4) ;
  \draw (6.1,3.4) .. controls (6.5,3.1) and (6.5,3.1) .. (7,3.4);

  \node at (4.2,3.2) [font=\tiny, scale=0.8, blue,above] {$\overline{\YY_2^{(0)}}$};
  
  \draw (6,0) .. controls (6.3,1) and (6.3,1).. (6,2.6) ;
  \draw (6,2.6) .. controls (5.7,3.6) and (5.7,3.6) .. (6,4.5);

  \node at (5.7,.5) [font=\tiny, scale=0.8, blue,below] {$\overline{\mathcal{Y}_3}$};

  \draw (6.5,0) .. controls (6.8,1) and (6.8,1).. (6.5,2.6) ;
  \draw (6.5,2.6) .. controls (6.2,3.6) and (6.2,3.6) .. (6.5,4.5);

  \node at (6.8,.5) [font=\tiny, scale=0.8, blue,below] {$\overline{\mathcal{Y}_4}$};
\end{tikzpicture}
\caption{Special Fiber of $\Y$}
\label{pic:Yspecialfiber}
\end{minipage}
\end{figure}
\label{example:allcases}
\end{ex3}
\subsection{Genus of $\overline{Y_t}$}
To get a complete understanding of the special fiber of $\Y$ we only
need to explain how to get the genus of each component from the
cluster and describe the component graph.
Keeping the previous notations,
let $\overline{Y_t}$ be a component of the special fiber of $\Y$ (we do
not assume that it is irreducible), above a component $X$ of $\X$,
corresponding to a cluster $\s$.

\begin{prop} Let $\widetilde{\s}_1, \ldots,\widetilde{\s}_N$ be the
  children of $\s$, let $a_i = |\widetilde{\s}_i|$ and let
  $d:=gcd(n,a_1, \ldots, a_N)$.  Then irreducible components of $\overline{Y_t}$
  have genus
\begin{align*}
\frac {1}{2d}  \left( n(N - 2) - \sum_{i = 1}^{N} \gcd(n,{a}_i) \right) + 1  +
  \begin{cases}
 0   & \textrm{ if } n \mid \sum_{i=1}^{N} {a}_i \\ \frac{n}{2d} - 
 \frac{\gcd(n,\sum_{i=1}^{N} {a}_i)}{2d}  & \textrm{ if } n\nmid \sum_{i=1}^{N} {a}_i
\end{cases}
\end{align*} 
\label{prop:genus}
\end{prop}
  \begin{proof}
    Since the genus of a curve equals that of its normalization, we
    can look at the components of $\overline{\YY_t}$. By
    Proposition~\ref{prop:numbercomponents}, we know that the
    components are given by an equation of the form
    $\overline{\YY_t^{\ell}}: y_t^{n/d} = \zeta_d^\ell c^{1/d}
    \prod_{i=1}^N (x_t-\alpha_i)^{a_i/d}$.

    If $\pi : X \to X'$ is a general degree $D$ map between
    non-singular curves, The Riemann-Hurwitz formula (see for example
    \cite[Corollary 2.4]{MR0463157}) implies that
  \[
2g(X)-2 = D(2g(X')-2) + \sum_P (e_P-1),
\]
where $g(X)$ denotes the genus of $X$ and $e_P$ denotes the
ramification degree of $P$. Taking $X = \overline{\YY_t^\ell}$ and $X'=\PP^1$, $g(X')=0$, $D=\frac{n}{d}$ and
  \begin{itemize}
  \item $e_P = 1$ for all points $P \neq \alpha_i$ and $P \neq \infty$,
    
  \item as mentioned before, each point $\alpha_i$ has ramification
    degree $\frac{n}{\gcd(n,a_i)}$ and there are $\frac{\gcd(n,a_i)}{d}$ points above it.
    
  \item If $n \mid \sum_{i=1}^N a_i = \deg(\overline{f_t(x_t)})$,
    $\infty$ is not ramified. Otherwise, it is a ramified point, with
    ramification degree $\frac{n}{\gcd(n,\deg(\overline{f_t(x_t)}))}$
    and $\frac{\gcd(n,\deg(\overline{f_t(x_t)}))}{d}$ points.
  \end{itemize}
   Then Riemann-Hurwitz gives that the genus of $\overline{\YY_t^{(\ell)}}$ equals
\[
  \frac {1}{2}  \left( \frac{n}{d}(N - 2) - \sum_{i = 1}^{N} \frac{\gcd(n,{a}_i)}{d} \right) + 1  +
  \begin{cases}
 0   & \textrm{ if } n \mid \sum_{i=1}^N a_i, \\
 \frac{n}{2d} - \frac{\gcd(n,\sum_{i=1}^N a_i)}{2d} & \textrm{ if } n\nmid \sum_{i=1}^N a_i.
\end{cases}
  \]
\end{proof}
\section{The Galois representation of $\C$}

\subsection{Computing the Galois representation over $K_2$}
Let $\Upsilon = (V,E)$ denote the dual graph of the special fiber of
$\overline{Y}$ (also referred as the graph of components in
\cite{wewers}); it is an undirected graph whose vertices $V$ are the
irreducible components of $\overline{Y}$. The set $E$ contains an
edge joining a pair of vertices for each intersection point of the
corresponding components. Under our hypothesis, the action of
$\Gal(\overline{k}/k)$ on the set $\overline{X}$ is trivial, but its
action on the set of irreducible components of $\overline{Y_t}$ (and
on $\Upsilon$) might not be. Let $\ell$ be a prime with
$\ell \nmid p$. Then by \cite[Lemma 2.7]{wewers} (see also
\cite[Corollary 1.6]{DokGalois}) it follows that as $\QQ_\ell[G_K]$-modules
\begin{equation}
  \label{eq:galrep}
  \cohoet(\overline{Y},\QQ_\ell) = \sum_{\tilde{Y} \in V} \cohoet(\tilde{Y},\QQ_\ell) \oplus {\text{H}}^1(\Upsilon,\ZZ)\otimes_{\ZZ}\QQ_\ell.
\end{equation}
If we consider the Picard group $\Pic^0(Y)$, it contains an abelian
part and a toric one (see for example \cite{Neron}, Example 8, page
246). The rank of the toric part equals the rank of
$\text{H}^1(\Upsilon,\ZZ)$, and its Galois representation consists of
Jordan blocks of size $2$ (see \cite[Proposition 3.5]{Grothendieck},
page 350). The action of $\Gal(\overline{K}/K)$ on $Y(\overline{K})$
extends to a semilinear action on the geometric points of
$\overline{Y}$ (see  \cite[Equation (2,18)]{Tim}, \cite[Corollary 1.6]{DokGalois} and page 13 of
\cite{MR2551757}). In particular, we have an isomorphism of
$G_K$-representations
\[
  V_{\ell}(\Pic^0(Y)) \simeq \left(\left(\text{H}^1(\Upsilon,\ZZ)\otimes_{\ZZ}\QQ_\ell\right) \otimes \Sp_2\right) \oplus \bigoplus_{\tilde{Y} \in V} V_\ell(\Pic^0(\tilde{Y})).
\]
Recall that the rank of $\coho^1(\Upsilon,\ZZ)$ equals $|E| - |V| + 1$
(because the graph is connected).
  \begin{remark}
    Since inertia acts trivially on $\text{H}^1(\Upsilon,\ZZ)$, the
    image of inertia consists of a matrix with
    $\rank_{\ZZ}\coho^1(\Upsilon,\ZZ)$ Jordan blocks of $2 \times 2$
    and the identity elsewhere. The values of the Frobenius element is
    given by its (permutation) action on the components.
\label{rem:inertia}
  \end{remark}

    \begin{thm}
    The rank of $\coho^1(\Upsilon,\ZZ)$ equals
    \[
      \sum_{\tilde{\s}}\gcd(n,|\tilde{\s}|)-\sum_{\s}\gcd(n,|\tilde{\s}_1|,\ldots,|\tilde{\s}_N|) + 1,
      \]
      where the first sum runs over all proper clusters except the
      maximal one, the second sum runs over all proper clusters, and
      the elements $\tilde{\s}_1,\ldots,\tilde{\s}_N$ denote the
      children of $\s$ (which might not be proper).
      \label{thm:cohomologyrank}
    \end{thm}
    \begin{proof}
      Recall that the rank of $\coho^1(\Upsilon,\ZZ)$ equals
      $|E| - |V| + 1$. The value $|V|$ (the number of irreducible
      components) equals the second term by
      Proposition~\ref{prop:numbercomponents}. The number of
      intersection points follows from the discussion after the same
      proposition, that states that
      $\#\varphi_t^{-1}(P)= gcd(n, v_P(f_t))$. Since
      $v_P(f_t) = |\tilde{\s}|$ the result follows.
    \end{proof}

    \begin{ex1}(Continued III). The graph of components (which can be
      read from Figure~\ref{fig:componentgraph}) is given in
      Figure~\ref{fig:graphcomp}. Its rank can be computed using the
      previous theorem, from the cluster description in
      Figure~\ref{fig:cluster1}, implying that
      $\coho^ 1(\Upsilon,\ZZ)$ has rank $7$ (which can be easily
      verified from the graph picture since the component graph
      $\Upsilon$ contains $9$ vertices and $15$ edges.). The advantage
      of Theorem~\ref{thm:cohomologyrank} is that we do not need to
      know the graph of components! (the cluster picture is
      enough).

      In particular, the image of inertia (of the Galois
      representation) equals $7$ Jordan blocks of the form
      $\left(\begin{smallmatrix} 1 & 1\\0 & 1\end{smallmatrix}
      \right)$, and the identity elsewhere.

      Recall by (\ref{eq:galrep}) that the Galois representation of
      $\C$ has two parts, one coming from the components and one
      coming from the graph of components. Let $\sigma$ denotes the
      Frobenius automorphism of $\Gal(K_2^{\text{ur}}/K_2)$.  If we
      want to understand its action on the component
      graph, we need to consider two different cases: if
      $\sqrt{-1} \in K_2$, then all components are defined over $K_2$
      and so are the intersection points, hence its action is trivial
      (and the $2 \times 2$ blocks correspond precisely to the
      classical Steinberg representation). However, if
      $\sqrt{-1} \not \in K_2$, then Frobenius interchanges the two
      components $\overline{\mathcal{Y}_1^{(1)}}$ and
      $\overline{\mathcal{Y}_1^{(2)}}$. A basis for the graph
      cohomology are the cycles:
\begin{multicols}{2}
      \begin{itemize}
      \item $e_1=\{,\overline{\mathcal{Y}_m^{(1)}},\overline{\mathcal{Y}_1^{(1)}},\overline{\mathcal{Y}_4},\overline{\mathcal{Y}_1^{(2)}}\}$,
      \item  $e_2=\{,\overline{\mathcal{Y}_m^{(1)}},\overline{\mathcal{Y}_1^{(1)}},\overline{\mathcal{Y}_3},\overline{\mathcal{Y}_1^{(2)}}\}$,
        
      \item $e_3=\{,\overline{\mathcal{Y}_m^{(1)}},\overline{\mathcal{Y}_1^{(1)}},\overline{\mathcal{Y}_2},\overline{\mathcal{Y}_1^{(2)}}\}$,
        
      \item $e_4=\{,\overline{\mathcal{Y}_m^{(1)}},\overline{\mathcal{Y}_1^{(1)}},\overline{\mathcal{Y}_m^{(2)}},\overline{\mathcal{Y}_5}\}$,

      \item $e_5=\{,\overline{\mathcal{Y}_m^{(1)}},\overline{\mathcal{Y}_1^{(2)}},\overline{\mathcal{Y}_m^{(2)}},\overline{\mathcal{Y}_5}\}$,
      \item $e_6=\{,\overline{\mathcal{Y}_m^{(1)}},\overline{\mathcal{Y}_1^{(1)}},\overline{\mathcal{Y}_m^{(3)}},\overline{\mathcal{Y}_5}\}$,

      \item $e_7=\{,\overline{\mathcal{Y}_m^{(1)}},\overline{\mathcal{Y}_1^{(2)}},\overline{\mathcal{Y}_m^{(3)}},\overline{\mathcal{Y}_5}\}$.
      \end{itemize}
\end{multicols}
Clearly $\sigma$ fixes $e_1, e_2, e_3$, while it interchanges
$e_4 \leftrightarrow e_5$ and $e_6 \leftrightarrow e_7$. Then
$\{e_1,e_2,e_3,e_4+e_5,e_6+e_7,e_4-e_5,e_6-e_7\}$ is a basis of
eigenvectors for $\sigma$ and the Galois representation on this basis
consists of $4$ copies of the Steinberg representation, and $3$ copies
of a twist of the Steinberg representation by the unramified quadratic
extension $K_2(\sqrt{-1})/K_2$.

Note that the sum of the genera of the components equals $12$, and
$7+12=19$ which is the genus of $\C$ (as it should be).
      
      \tikzstyle{vertex}=[circle,fill=black!25,minimum size=20pt,inner
      sep=0pt] \tikzstyle{selected vertex} = [vertex, fill=red!24]
      \tikzstyle{edge} = [draw,thick,-] \tikzstyle{edge2} =
      [draw,thick,-,red] \tikzstyle{edge3} = [draw,thick,-,blue]
      \tikzstyle{weight} = [font=\small] \tikzstyle{selected edge} =
      [draw,line width=5pt,-,red!50] \tikzstyle{ignored edge} =
      [draw,line width=5pt,-,black!20]
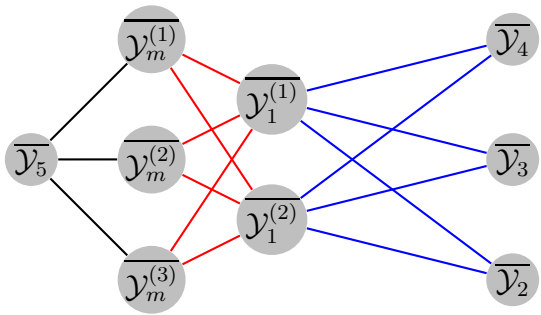
\begin{figure}[h]
\begin{tikzpicture}[scale=1.6, auto,swap]
  \foreach \pos/\name/\lbl in {{(0,3)/a/{\overline{\YY_5}}},
    {(1,4)/b/{\overline{\YY_m^{(1)}}}}, {(1,3)/c/{\overline{\YY_m^{(2)}}}}, {(1,2)/d/{\overline{\YY_m^{(3)}}}},
    {(2,3.5)/e/{\overline{\YY_1^{(1)}}}}, {(2,2.5)/f/{\overline{\YY_1^{(2)}}}},
    {(4,2)/g/{\overline{\YY_2}}}, {(4,3)/h/{\overline{\YY_3}}}, {(4,4)/i/{\overline{\YY_4}}}}
  \node[vertex] (\name) at \pos {$\lbl$};

\foreach \source/ \dest in {b/a, c/a,d/a}
        \path[edge] (\source) -- node[weight] {} (\dest);
\foreach \source/ \dest /\weight in {e/b, e/c,e/d,
                                     f/b,f/c,f/d}
\path[edge2] (\source) -- node[weight] {} (\dest);
\foreach \source/ \dest /\weight in {g/e,g/f,
                                     h/e,h/f,
                                     i/e,i/f}
\path[edge3] (\source) -- node[weight] {} (\dest);
\end{tikzpicture}
\caption{The component graph $\Upsilon$.}
\label{fig:graphcomp}
\end{figure}

\end{ex1}
\begin{ex3}[Continued] From the cluster picture (see
  Figure~\ref{fig:cluster5}) and Theorem~\ref{thm:cohomologyrank} we
  get that $\coho^ 1(\Upsilon,\ZZ)$ has rank $14$, hence the image of
  inertia equals $14$ Jordan blocks of size $2 \times 2$.  The
  component graph $\Upsilon$ contains $14$ vertices and $27$ edges
  (which can be read from Figure~\ref{pic:Yspecialfiber}). The sum of
  the genera of the components equals $20$, and $20+14=34$ which is
  the genus of $\C$.

  A similar analysis as the one made in the previous example can be
  used to determine the graph component representation.  For the
  components $\overline{\YY_m^{(i)}}$ and $\overline{\YY_5^{(j)}}$ be
  defined over $K$ we need $\sqrt{-2}$ be in
  $K_2$. Start supposing this is the case. The group $G_{K_2}$ fixes the
  vertices of the graph, but still might not fix the edges
  (corresponding to the intersection points). The intersection of
  $\overline{\YY_m^{(i)}}$ with $\overline{\YY_5^{(j)}}$ correspond to
  two points with coordinates in $K_2(\sqrt{-2\zeta_3^i})$ which is
  fixed by $G_{K_1}$ under our hypothesis, but the intersection of
  $\overline{\YY_m^{(i)}}$ with $\overline{\YY_2^{(i)}}$ correspond to
  two points with coordinates in $K_2(\sqrt{-\zeta_3^i})$.
    
  In particular, if $\sqrt{-1}$ also belongs to $K_2$, the Galois
  representation attached to $\C$ decomposes as a direct sum of
  dimensions $8 + 8 + 8 + 2 + 2 + 4 + 4 + 4$ (corresponding to the
  curves $\overline{\YY_1}$, $\overline{\YY_3}$, $\overline{\YY_4}$,
  $\overline{\YY_5^ {(0)}}$, $\overline{\YY_5^ {(1)}}$,
  $\overline{\YY_6}$, $\overline{\YY_7}$ and $\overline{\YY_8}$
  respectively) and $14$ blocks where the action of Frobenius is
  trivial and a generator of inertia acts as
  $\left( \begin{smallmatrix} 1 & 1 \\ 0 & 1 \end{smallmatrix}
  \right)$ (corresponding to the Steinberg representation). However,
  if $\sqrt{-1} \not \in K_2$, then $G_{K_2}$ permutes the two edges joining
  the vertices $\overline{\YY_m^{(i)}}$ and $\overline{\YY_2^{(i)}}$
  . This implies that three of the fourteen blocks have Frobenius
  acting by $-1$, while the other eleven ones keep the trivial
  action. Note that since $p \nmid 6$ all the extensions are
  unramified, so the image of inertia does not change.

  Suppose that $\sqrt{-2} \not \in K_2$, then $G_{K_2}$
  permutes the two components $\overline{\YY_5^ {(0)}}$ and
  $\overline{\YY_5^ {(1)}}$ (with its respective intersection points),
  which induces another involution on the graph of components. In this
  case, if $\sqrt{-1} \in K_2$, five Jordan blocks have Frobenius acting
  by $-1$ while the other seven ones have trivial action while if
  $\sqrt{-1} \not \in K_2$, eight blocks have Frobenius acting by $-1$
  and by $1$ on the other five blocks. The remaining cases can be
  studied similarly.
  \end{ex3}

\begin{remark}
  We want to emphasize that when restricting the Galois representation
  to $K_2$, the image of inertia depends only on the components graph,
  which can easily be read from the cluster picture. The same is true over $K$ considering \emph{weighted clusters} that will be introduced in the next section.
\end{remark}

  \section{The Galois representation over $K$}
\label{section:repoverK}
  \subsection{Weighted Cluster}
\label{section:weightedcluster}
Recall from Proposition~\ref{prop:numbercomponents} that the
components $\overline{Y_t}$ correspond to equations where the leading
coefficient of $f(x)$ might not be $1$. This corresponds to a
``twist'' of the representation. To determine whether the involved
character is ramified or not it is important the notion of \emph{weighted
  clusters}.

\begin{definition}
  Let $\mathfrak{s}$ be a proper cluster (i.e.
  $\mathfrak{s} \neq \mathcal{R}$). Define its \emph{relative diameter} (that will be denoted $d_{\mathfrak{s}}$) by
\begin{equation}\notag
d_{\mathfrak{s}} = \mu_{\mathfrak{s}} - \mu_{P(\mathfrak{s})}
\end{equation}
where $P(\mathfrak{s})$ denotes the parent of $\mathfrak{s}$.

\end{definition}

Following\cite{Tim}, in a cluster picture we include the relative
diameters as follows: in a maximal cluster include a subscript
denoting its diameter; for all other clusters include a
subscript given by their relative diameter (that is the difference between
their diameters and that of their parent cluster).

\begin{ex1} Recall that Table~\ref{table:radii} gives the 
diameters $\mu_{\sm} = 0$, $\mu_{\s_1} = \mu_{\s_5} = 1$, $\mu_{\s_2}
=\mu_{\s_3} = \mu_{\s_4}=2$. Then their relative diameter equal
\begin{align*}
&d_{\s_1} = \mu_{\s_1}-\mu_{\sm} = 1, \; d_{\s_5} = \mu_{\s_5} - \mu_{\sm},\\
& d_{\s_2}= \mu_{\s_2}-\mu_{\s_1} = 1, d_{\s_3} = \mu_{\s_3}-\mu_{\s_1} = 1, \; d_{\s_4} = \mu_{\s_4}-\mu_{\s_1} = 1. 
\end{align*}
giving the following weighted cluster.
$$
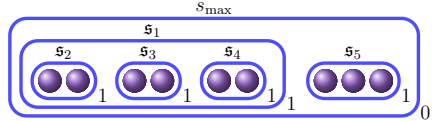
\scalebox{1.6}{\clusterpicture
 \Root {1} {first} {r1};
  \Root {} {r1} {r2};
  \Root {3} {r2} {r3};
  \Root {} {r3} {r4};
  \Root {3} {r4} {r5};
  \Root {} {r5} {r6};
  \Root {5} {r6} {r7};
  \Root {} {r7} {r8};
  \Root {} {r8} {r9}
   \ClusterLDName c1[][1][\s_2] = (r1)(r2);
  \ClusterLDName c2[][1][\s_3] = (r3)(r4);
   \ClusterLDName c3[][1][\s_4] = (r5)(r6);
   \ClusterLDName c4[][1][\s_1] = (c1)(c2)(c3)(c1n)(c2n)(c3n);
  \ClusterLDName c5[][1][\s_5] = (r7)(r8)(r9);
  \ClusterLDName c6[][0][\sm] = (c4)(c5)(c4n)(c5n);
\endclusterpicture}
$$
\end{ex1}
Given $\mathfrak{s}_1, \mathfrak{s}_2$ two clusters (or roots) let
$\mathfrak{s}_1 \wedge \mathfrak{s}_2$ denote the smallest cluster
that contains both of them. For instance, in the previous example
$0 \wedge 1 = \mathcal{R}$, $\s_2 \wedge \s_3 = \s_1$,
$\s_2 \wedge p+p^2 = \s_1$. Keep the notation of the previous
sections, and let $e_t$ be the valuation of the content of the
polynomial $f(x_t)$ (in particular $e_t = v(c_t)$).
\begin{prop} If $t \in T$ corresponds to a component of the special
  fiber of $\overline{X}$ associated to a cluster $\mathfrak{s}$, the
  content valuation of the polynomial $f(x_t)$ equals
\begin{equation}\notag
e_t = \sum_{r \in \R} \mu_{r \wedge \mathfrak{s}}. 
\end{equation}
\label{proposition:weightedcluster}
\end{prop} 

\begin{proof}
  Recall that if $t$ corresponds to a cluster
  $\mathfrak{s}= D(\alpha, \mu_{\mathfrak{s}})$ then
  $x = \pi^{\mu_{\mathfrak{s}}}x_t + \alpha$ where
  $\alpha \in \mathfrak{s}$ and
\begin{align*}
f(x_t) = \prod_{r \in \R} (\pi^{\mu_{\mathfrak{s}}}x_t + \alpha - r).
\end{align*} 
Each factor $(\pi^{\mu_{\mathfrak{s}}}x_t + \alpha - r)$ has
content valuation
$\min\lbrace \mu_{\mathfrak{s}}, v(\alpha - r) \rbrace$ contributing to the
content valuation $c_t$ of $f(x_t)$. Consider the following two cases:
\begin{itemize}
\item If $r \in \s$ then $\min\lbrace \mu_{\mathfrak{s}}, v(\alpha - r) \rbrace = \mu_\s = \mu_{\s \wedge r}$.
  
\item Otherwise, $\min\lbrace \mu_{\mathfrak{s}}, v(\alpha - r) \rbrace = v(\alpha -r) = \mu_{\s \wedge r}$ as well.
\end{itemize}
Then the formula follows.
\end{proof}

\subsection{Decomposing the representation of $\C$}
\label{subsection:decomposing}
A good reference for details on this section is \cite{MR790953}.  Let
$G$ denote the group $\mu_n$ of $n$-th roots of unity, whose group
algebra equals
\begin{equation}
  \label{eq:endomorphism}
\QQ[G] = \QQ[t]/(t^n-1)\simeq \prod_{d \mid n}\QQ[t]/\phi_d(t),  
\end{equation}
where $\phi_d(t)$ denotes the $d$-th cyclotomic polynomial (with
complex roots the primitive $d$-th roots of unity). Fix $\zeta_n$ a
primitive $n$-th root of unity (which belongs to $K$). The group $G$
acts on $\C$ via $t\cdot(x,y) = (x,\zeta_n y)$. This action extends to
an action of $\QQ[G]$ in
$\Aut^0(\Jac(\C)) := \Aut(\Jac(\C))\otimes_\ZZ \QQ$. Let
$V_\ell(\Jac(\C))$ denote the $\QQ_\ell$ Tate module $T_\ell(\Jac(\C))\otimes_{\ZZ_\ell}\QQ_\ell$. The
natural injective morphism
$\End(\Jac(\C))\otimes \QQ_\ell \hookrightarrow \End(V_\ell(\Jac(\C))$
gives an action of $\QQ_\ell[G]$ on $V_\ell(\Jac(\C))$.

If $H$ is a subgroup of $G$ (corresponding necessarily to the group of
$d$-th roots of unity for some $d \mid n$) we have a natural
surjective map $\pi_H : \C \to \C/H:=\C_H$. In particular, if
$H = H_{n/d}$ (corresponding to $\mu_{n/d}$), denote the quotient curve $\C/H$ by
$\C_d$, with equation:
\begin{equation}
  \label{eq:C_d}
  \C_d: y^{d}=f(x).
\end{equation}
The quotient map is given explicitly by $\pi_d(x,y) = (x,y^{n/d})$ (an $n/d$ to $1$
map). This induces two morphisms between $\Jac(\C)$ and $\Jac(\C_d)$
namely the push-forward $\pi_*:\Jac(\C) \to \Jac(\C_d)$ and the
pullback $\pi_d^*:\Jac(\C_d) \to \Jac(\C)$ whose kernel is contained
in the $n/d$-torsion of $\Jac(\C_d)$. Let $A_d$ denote the connected
component of $\ker(\pi_*)$. For any prime $\ell$ we get an
injective morphism on the $\QQ_\ell$-Tate modules
$\pi^*_\ell:V_\ell(\Jac(\C_d)) \to V_\ell(\Jac(\C))$ and 
\[
V_\ell(\Jac(\C))= V_\ell(A_d) \oplus \pi^*_\ell(V_\ell(\Jac(\C_d))).
  \]
  The group $\mu_{d}$ acts on $\C_d$. For any
  $\alpha \in \QQ[\mu_{d}]$ let
  $\pi^*(\alpha) = \frac{d}{n}(\pi_d^* \circ \alpha \circ
  \pi_*)$. Then (see the proof of Proposition 2 in \cite{MR790953})
  $\pi^*(\alpha)|_{V_\ell(A)} = 0$ and
  $\pi^*(\alpha)|_{\pi^*_\ell(V_\ell(\Jac(\C_d)))} = \alpha$.

  In particular, the Galois representation attached to the curve
  $y^n=f(x)$ contains for each $d \mid n$ what might be called a
  \emph{$d$-new} part coming from the curve $y^d=f(x)$ and
  $V_\ell(\Jac(\C)) = \bigoplus_{d \mid n}
  V_\ell(\Jac(\C_{d}))^{d\text{-new}}$.  Then in the
  decomposition~(\ref{eq:endomorphism}) the action of the group
  algebra $\QQ_\ell[t]/\phi_d(t)$ on $V_\ell(\Jac(\C))$ is non-trivial
  precisely in the subspace corresponding to
  $V_\ell(\Jac(\C_d))^{d\text{-new}}$.

  \begin{ex*}
    Suppose that $n = p \cdot q$ with $p,q$ distinct prime
    numbers. Then 
    \[
V_\ell(\Jac(\C)) =  V_\ell(\Jac(\C))^{pq\text{-new}} \oplus V_\ell(\Jac(\C_p)) \oplus V_\ell(\Jac(\C_q)),
\]
where $V_\ell(\Jac(\C))^{pq\text{-new}}=V_\ell(A_p) \cap V_\ell(A_q)$.
The group algebra $\QQ[t]/\phi_{pq}(t)$ acts non-trivially on the
first summand, $\QQ[t]/\phi_p(t)$ on the second and $\QQ[t]/\phi_q(t)$
on the third one.
  \end{ex*}
An explicit description of $V_\ell(\Jac(\C))^{n\text{-new}}$ can be given  as virtual representations using the inclusion-exclusion principle. 

\begin{remark}
  The contribution from $H = G$ in the above formula is trivial, as it
  corresponds to a genus $0$ curve. This is the reason why one can
  remove the term with $d=1$ in (\ref{eq:endomorphism}).
\end{remark}

\subsection{Twisting} Let $c \in K$ be a non-zero element,
$f(x) \in K[x]$ and consider the following two curves:
\[
  \C:y^n = f(x),
  \]
and
\[
\C':y^n =c\cdot f(x).
  \]
  It is clear that they become isomorphic over the (abelian) extension
  $K(c^{1/n})$, in particular, they are a twist of each other. 

  \smallskip
  
  {\noindent \bf Problem:}  what is the relation between the
Galois representations of $\C$ and that of $\C'$?

\smallskip

This problem appears in different contexts. For example, if we start
with a monic polynomial $f(x)$ (which we assumed was the case) and
want to consider a general polynomial (with all roots in $K$), we need
to understand twists. Also while computing the semistable model (in
Proposition~\ref{prop:numbercomponents}) the components involve taking
twists by a $d$-th root of $c_t$. The hypothesis $\zeta_n \in K$
implies that the extension associated to the twist is a Galois one,
hence the extension $K(\sqrt[n]{c})/K$ corresponds to a Hecke character.

The problem is probably known to experts (as happens for example in
the case of an elliptic curve twisted by a quadratic character, or an
elliptic curve with CM by $\ZZ[\zeta_3]$ while twisted by a cubic or
sextic character) but we did not find a good reference in the
literature, so we briefly explain it.

Since both curves become isomorphic over the extension
$K[\sqrt[n]{c}]$ their representations must be related by some sort of
twist. More concretely, if we base extend $\C$ to $L=K[\sqrt[n]{c}]$
(let $\C_L$ denote such curve) and we do the same to $\C'$ then both
curves become isomorphic, hence their Galois representations are the
same. Recall that the representation
attached to $\Jac(\C)$ and $\Res_{K}(\Jac(\C_L))$ are related via
\begin{equation}
  \label{eq:extres}
  V_\ell(\Res_{K}(\Jac(\C_L))) = \bigoplus_{\chi} V_\ell(\Jac(\C))\otimes \chi,
\end{equation}
where $\chi$ ranges over the characters of the (abelian) group $\Gal(L/K)$.

However this picture is a little misleading, as it is not true that
$V_\ell(\Jac(\C'))$ equals $V_\ell(\Jac(\C)) \otimes \chi$ (for some
character $\chi$) in general (note that the latter does not have the
right determinant for example). What happens is that $V_\ell(\Jac(\C))$
(respectively $V_\ell(\Jac(\C'))$) has a decomposition (as explained
in Section~\ref{section:weightedcluster}) of the form:
\[
  V_\ell(\Jac(\C)) = \bigoplus_{d \mid n}V_\ell(\Jac(\C_d))^{d\text{-new}}.
\]
Recall that $\QQ_\ell[t]/\phi_d(t)$ acts on
$V_\ell(\Jac(\C_d))^{d\text{-new}}$, hence the latter admits a
decomposition in terms of the action of the $d$-th roots of unity.
Concretely. pick a basis for the order $\ell^n$ points (for each
n) as a $\ZZ[\zeta_d]$-modules instead of taking one as a
$\ZZ$-module. Once a $d$-th root of unity (say $\zeta_d$) is chosen
inside the automorphism group of $\Jac(\C)$, we get the decomposition
\begin{equation}
  \label{eq:unityrootsdecomposition}
  V_\ell(\Jac(\C_d))^{d\text{-new}} = \mathop{\bigoplus_{i=1}^{d}}_{\gcd(i,d)=1}V_\ell^{(i)}(\Jac(\C_d))^{d\text{-new}},
\end{equation}
as $\QQ_\ell[\Gal_{K}]$-modules where $t$ acts on
$V_\ell^{(i)}(\Jac(\C_d))^{d\text{-new}}$ as $\zeta_d^i$. There is an
explicit character $\chi$ (depending on $d$ and $c$) such that
\begin{equation}
  \label{eq:twist}
  V_\ell^{(i)}(\Jac(\C_d))^{d\text{-new}} \simeq V_\ell^{(i)}(\Jac(\C_d'))^{d\text{-new}} \otimes \chi^i.  
\end{equation}
To describe it fix $\zeta_n$ an $n$-th root of
unity in $K$. Such a choice determines an element (abusing
notation) $\zeta_n\in \End(\Jac(\C))$ and an element $\zeta_n$
(abusing notation again) in $V_\ell(\Jac(\C))$ (its image under the
map
$\End(\Jac(\C))\otimes \ZZ_\ell
\hookrightarrow \End(T_\ell(\Jac(\C)))$).
  \begin{lemma}
    Let $L=K[\sqrt[n]{c}]$, let $r=[L:K]$, and let
    $V_\ell^{(i)}(\Jac(\C_d))^{d\text{-new}}$ denote the subspaces in the decomposition
    (\ref{eq:unityrootsdecomposition}). Let $\sigma \in \Gal(L/K)$ be
    the generator sending $\sqrt[n]{c}$ to $\zeta_n^{n/r} \sqrt[n]{c}$
    and let $\chi:\Gal(L/K) \to \overline{\QQ_\ell}$ denote the
    character sending $\sigma$ to $\zeta_n^{n/r}$.  Then for all
    $1 \le i \le n$, prime to $n$ we have
    \[
       V_\ell^{(i)}(\Jac(\C_d))^{d\text{-new}} \simeq V_\ell^{(i)}(\Jac(\C_d'))^{d\text{-new}}\otimes \chi^i.
  \]
\label{lemma:twist}
  \end{lemma}
  \begin{proof}
    Let $\varphi:\C \to \C'$ be the map
    $\varphi(x,y) = (x,\sqrt[n]{c}\, y)$ and let
    $\tilde{\sigma} \in \Gal_{K}$ be such that its restriction to
    $K[\sqrt[n]{c}]$ equals $\sigma$. We claim that
    \begin{equation}
      \label{eq:relation}
\tilde{\sigma}\circ \varphi = \zeta_n^{n/r} \cdot \varphi \circ \tilde{\sigma}.      
    \end{equation}
    If we compute both maps on a point $(x,y)$, the left hand side equals
    \[
      (\tilde{\sigma}(\sqrt[n]{c})\cdot\tilde{\sigma}(x),\tilde{\sigma}(y))=(\zeta_n^{n/r}\tilde{\sigma}(x),\tilde{\sigma}(y)),
    \]
    which clearly equals the right hand side hence the claim. The
    result follows easily from (\ref{eq:relation}) recalling that on
    $V_\ell^{(i)}(\Jac(\C_d)^{d\text{-new}})$ the element $t$ acts by
    $(\zeta_n^{n/r})^i$.
  \end{proof}

  Note that there are two different types of twisting affecting the
  Galois representation, and the L-series $p$-th factor, namely
  unramified and ramified ones. Unramified twists already appeared
  while computing the Galois representation of Example~\ref{ex:1} (in
  page $17$). They affect the value of Frobenius, but does not change
  the image of inertia. To compute such twists on the components of
  positive genus, it is probably easier to compute the number of
  points of the components of the twisted curve rather than assuming
  the polynomial $f(x)$ is monic and then computing the twist (see
  \cite{Drew} for a fast method to count the number of points).

  Ramified twists on the contrary affects the image of inertia. The
  use of weighted cluster is very handful to distinguish whether the
  twist by $c_t$ appearing on the components of
  Proposition~\ref{prop:numbercomponents} are ramified or not.
  \begin{prop}
    Let $t$ be a component of $X$ corresponding to a cluster
    $\s$. Then the components of $\YY_t^{(\ell)}$ are ramified twists
    of a non-singular superelliptic curve precisely when $d \nmid e_t$.
  \end{prop}
  \begin{proof}
    Follows from the fact that $e_t$ is the valuation of $c_t$ (see
    Proposition~\ref{prop:numbercomponents} for the notation).
  \end{proof}
  In particular, Proposition~\ref{proposition:weightedcluster} shows
  how to verify this condition from the weighted cluster picture. Note
  that for each $d \mid n$, if $d \nmid e_t$ the image of inertia in
  abelian part of the $d$-new part is given by $t$-copies of
  \[
    \mathop{\bigoplus_{i=1}^{d}}_{\gcd(i,d)=1}\chi^i,
  \]
  where $\chi$ is the
  ramified character corresponding  to the extension
  $K(\sqrt[d]{p^{e_t}})/K$ and $t = \frac{2g(y_t^d=f_t(x_t))}{\phi(d)}$.
  \begin{ex1} Recall the weighted cluster picture:
    $$
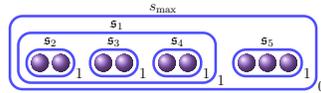
\scalebox{1.2}{\clusterpicture
 \Root {1} {first} {r1};
  \Root {} {r1} {r2};
  \Root {3} {r2} {r3};
  \Root {} {r3} {r4};
  \Root {3} {r4} {r5};
  \Root {} {r5} {r6};
  \Root {5} {r6} {r7};
  \Root {} {r7} {r8};
  \Root {} {r8} {r9}
   \ClusterLDName c1[][1][\s_2] = (r1)(r2);
  \ClusterLDName c2[][1][\s_3] = (r3)(r4);
   \ClusterLDName c3[][1][\s_4] = (r5)(r6);
   \ClusterLDName c4[][1][\s_1] = (c1)(c2)(c3)(c1n)(c2n)(c3n);
  \ClusterLDName c5[][1][\s_5] = (r7)(r8)(r9);
  \ClusterLDName c6[][0][\sm] = (c4)(c5)(c4n)(c5n);
\endclusterpicture}
$$
Proposition~\ref{proposition:weightedcluster} gives that: $e_{\sm}=0$,
$e_{\s_1}=6$, $e_{\s_2}=e_{\s_3}=e_{\s_4}=8$ and $e_{\s_5}=3$. This
implies that no ramified twist is involved on $\overline{Y_1^{(l)}}$
(its components are genus $1$-curves), while the curves
$\overline{Y_2}$, $\overline{Y_3}$ and $\overline{Y_4}$ (all of them
of genus $2$) involve a ramified twist $\chi$ corresponding to the
extension $\QQ_p(\sqrt[3]{p})/\QQ_p$. Such curves have a $2$-new part (of
genus $0$), a $3$-new part (of genus $1$) giving the representation of
inertia $\chi \oplus \chi^2$ and a $6$-new part (also of genus $1$) giving the same
representation of inertia.

Regarding the component $\overline{Y_5}$ (of genus $4$), let $\psi$ be
the character attached to the representation $\QQ_p(\sqrt{p})/\QQ_p$.
The curve has a $2$-new part of genus $1$, giving the representation
$\psi \oplus \psi$ (since $2 \nmid e_{\s_5}$); has a $3$-new part
(also of genus $1$) which does not involve any twist (as $3 \mid 3$)
hence inertia acts trivially in this $2$-dimensional part; and a
$6$-new part (of dimension $4$) where inertia acts via the quadratic
character $\psi$.

To understand the toric part, we need to understand the action of
$\Gal(\QQ_p(\sqrt[3]{p})/\QQ_p)$ on the component graph. The way to
compute this action is very well explained in \cite{DokGalois} (see
Examples 1.9 and 1.11). Concretely, it is given by what they call the
``lift-act-reduce'' procedure. Let
$\sigma \in \Gal(\overline{\QQ_p}/\QQ_p)$ and $(\bar{x},\bar{y})$ a
point on $\overline{Y_m}$. Any lift corresponds to a point
$(\tilde{x},tilde{y})$ (on the curve $\C$), hence the reduction of its
action corresponds to the point $(\sigma(\bar{x}),\sigma(\bar{y})$. In
particular, it fixes the components $\overline{Y_m^{(i)}}$ and its
intersection points as well. The same computation for the components
of $\overline{Y_2}$ gives the action:
\[
(\bar{x},\bar{y}) \to \left(\frac{\tilde{x}-1}{p},\frac{\tilde{y}}{p}\right) \to \left(\frac{\sigma(\tilde{x})-1}{p},\frac{\sigma(\tilde{y})}{p}\right) \to (\sigma(\bar{x}),\sigma(\bar{y})).
\]
In particular it also fixes the three components as well as the
intersection points. A similar computation proves the same result for
the components of $\overline{Y_5}$, hence the image of inertia is the
same over $\QQ_p$ than over $\QQ_p(\sqrt[3]{p})$ in this particular example.

\end{ex1}

\bibliographystyle{alpha}
\bibliography{biblio}
\end{document}